\documentclass[opre,nonblindrev]{informs4}
\OneAndAHalfSpacedXII
\usepackage{makecell}
\usepackage{bm}
\usepackage{endnotes}
\usepackage{booktabs}
\usepackage{bbm}
\usepackage{scalerel}
\usepackage{setspace}
\usepackage{float}
\let\footnote=\endnote

\usepackage{hyperref}
\makeatletter
\DeclareFontShape{OT1}{ptm}{m}{scit}{<->ssub * ptm/m/sc}{}
\makeatother
\usepackage{natbib}
 \bibpunct[, ]{(}{)}{,}{a}{}{,}%
 \def\bibfont{\small}%
\TheoremsNumberedThrough     % Preferred (Theorem 1, Lemma 1, Theorem 2)
\ECRepeatTheorems

\EquationsNumberedThrough    % Default: (1), (2), \cdots
\makeatletter

\let\informs@proof\proof
\let\informs@endproof\endproof

\RenewDocumentEnvironment{proof}{o}{
  \IfValueTF{#1}
    {\informs@proof{#1}}   % Proof with custom name
    {\informs@proof{Proof}} % Default "Proof"
}{
  \informs@endproof
}

\makeatother

\usepackage{makecell, multirow}
\usepackage{tablefootnote}
\usepackage{subfigure}
\usepackage{enumitem}
\usepackage{fix-cm}
\usepackage{bbm}
\usepackage{etoolbox}  % Need this for 
\makeatletter
\newcommand{\appendixtableofcontents}{%
  \section*{Appendices}
  \@starttoc{app}%
}

\let\originaladdcontentsline\addcontentsline

\newcommand{\appendixaddtocline}[3]{%
  \originaladdcontentsline{toc}{#2}{#3}% Write to main TOC
  \ifstrequal{#2}{section}{%
    \phantomsection
    \originaladdcontentsline{app}{section}{#3}% Write to appendix TOC
  }{%
    \ifstrequal{#2}{subsection}{%
      \phantomsection
      \originaladdcontentsline{app}{subsection}{#3}%
    }{}%
  }%
}

\AtBeginEnvironment{APPENDICES}{%
  \let\addcontentsline\appendixaddtocline
}

\hypersetup{bookmarksdepth=subsection}
\makeatother
\MANUSCRIPTNO{}
\usepackage[ruled,vlined, noend,linesnumbered]{algorithm2e}

\usepackage{amsmath,amssymb,mathtools}
\usepackage{fullpage}
\usepackage{hyperref}
\usepackage{enumitem}
\usepackage{graphicx}
\usepackage{xcolor}
\usepackage{algpseudocode}
\usepackage{todonotes}

\newcommand{\ordscore}{\preceq_{\mathrm{score}}}

\newcommand{\Calls}{\mathcal{C}}                 % (i,r)
\newcommand{\CallsT}{\mathcal{C}^{\mathrm{time}}}% (i,r,t)

\newcommand{\ag}[1]{i_{#1}}      % agent of e
\newcommand{\pr}[1]{r_{#1}}      % price of e
\newcommand{\ti}[1]{t_{#1}}      % time of e (only meaningful for time-indexed)

\newcommand{\Order}{\preceq}                      % generic global order
\newcommand{\orddown}{\preceq_{\downarrow}}       % decreasing-price order
\newcommand{\Pred}[3]{\mathsf{Pred}_{#1}\!\left(#2,#3\right)}  % Pred_order(e,S)
\newcommand{\Succ}[3]{\mathsf{Succ}_{#1}\!\left(#2,#3\right)}  % Succ_order(e,S)

\newcommand{\Rev}[4]{\mathsf{Rev}_{#1}\!\left(#2;#3,#4\right)}

\newcommand{\Tmax}{T_{max}}   % upper bound on time indices
\newcommand{\SelG}{Sel_{G}}

\newcommand{\prune}{\operatorname{prune}}
\newcommand{\sel}{\operatorname{sel}}

\algnewcommand{\Input}{\item[\textbf{Input:}]}
\algnewcommand{\Output}{\item[\textbf{Output:}]}

\newcommand{\Plan}{\pi}

\newcommand{\RevV}[3]{\mathrm{Rev}_{#1}\!\left(#2;\,#3\right)} % \RevV{n}{\pi}{v}

\newcommand{\RevN}[2]{\mathrm{Rev}_{#1}\!\left(#2\right)}      % \RevN{n}{\pi}

\newcommand{\RevUnk}[1]{\mathrm{Rev}\!\left(#1\right)}         % \RevUnk{\pi}

\newcommand{\OPT}{OPT}

\begin{document}
%%%%%%%%%%%%%%%%

%\setcounter{page}{1}
%\DOI{10.1287/xxxx.0000.0000}%

% \RUNAUTHOR{Slippery and Arinella}


% Full title. Sampleo
% \TITLE{Bundling Information Goods of Decreasing Value}
\TITLE{Sequential Pricing with Deadlines and Correlated Buyers}

\ARTICLEAUTHORS{%
\AUTHOR{Shruti Penumatsa, Rajan Udwani}
\AFF{Department of Industrial Engineering and Operations Research,
UC Berkeley, Berkeley, CA 94720,
\EMAIL{shruti\_penumatsa@berkeley.edu}, \EMAIL{rudwani@berkeley.edu}}
} % end of the block

\ABSTRACT{%
We study sequential posted pricing for selling a single item under an exogenous
deadline: a seller makes take-it-or-leave-it price offers to buyers one at
a time, stopping when some buyer accepts or the selling opportunity expires. We consider both deterministic
deadlines and random deadlines drawn from a known distribution, and allow buyer
valuations to be either independent or arbitrarily correlated. Despite its
practical relevance, this collection of deadline-constrained sequential pricing
models has received limited algorithmic study beyond special cases. We develop the first approximation algorithms for revenue maximization across
these settings.

}%

\maketitle

% Paper body
\section{Introduction}
Posted prices and auctions are the dominant paradigms for selling items and have
been extensively studied.  Yet in many high-value transactions, sales do not
occur via a single centralized mechanism: negotiation and diligence often
proceed with one buyer at a time, due to exclusivity provisions and credibility
concerns.  As a result, sellers commonly
make sequential take-it-or-leave-it offers and stop as soon as some buyer
accepts.

For a concrete example, consider the sale of a home through a pre-market process \citep{zillow_off_mls}.
Before listing on the public market, a seller (often through an agent) may
privately approach a small set of prospective buyers sequentially with
take-it-or-leave-it price offers.  If one buyer accepts, the home is sold and the
process ends; otherwise, the seller continues to the next buyer.  This private
sales phase is naturally governed by a deadline: if no agreement is reached
within a limited window, the property is listed publicly, at which point the
private sequential-offer procedure effectively terminates.

A second example comes from mobile in-app advertising with ad mediation.
When a user opens an app, a publisher has an opportunity to serve an
advertisement for that impression.  Rather than running a single centralized
auction across all demand sources, mediation platforms commonly allocate the
impression via a \emph{multi-call auction} \citep{display-ads-open-problems-talk}: they query ad networks sequentially
with reserve prices and stop as soon as some network returns a bid that clears
the reserve, or the opportunity expires because the user disengages.

A unifying feature of these settings is an exogenous deadline: the seller
cannot continue making offers indefinitely. In the housing example, the private
pre-market phase ends when the property transitions to a different sales channel,
such as a public listing.  In the advertising example, the opportunity ends when
the user’s attention window closes, for instance because the user scrolls away and
leaves the app, so there is only limited time to run the multi-call auction.
Under such a constraint, the seller’s problem is not only which prices to
offer, but also how to allocate a limited number of attempts over time to
the right buyer at the right price before the opportunity expires.  Depending on
the application, the number of offers that the seller can make may be known in advance or random.

These environments also naturally feature correlated demand.  In housing,
buyers’ willingness to pay may be shaped by shared market conditions and common
shocks.  In ad mediation, correlation may arise due to various reasons: (i) the
mediator may query the same ad network multiple times with different
reserves, creating dependence across repeated interactions with a single demand
source\footnote{Our formal model allows at most one offer to each buyer, but correlated buyer valuations capture the setting with repeated or multiple offers as a special case (see Section \ref{subsec:multi-call-def}).}; and (ii) the same advertisers may participate across multiple ad
networks, inducing dependence in bids across networks for the same impression.
These considerations motivate allowing valuations 
to be drawn from an arbitrary joint distribution across buyers.

We therefore study revenue maximization under sequential
take-it-or-leave-it offers with an exogenous deadline and buyer correlation. Our goal is to take a first step toward understanding the computational
aspects of deadline-constrained sequential sales problems.  Accordingly, we abstract away from several other interesting features of these environments, such as, strategic delay, learning unknown distributions, and dynamic effects (e.g., time-evolving correlations), and view these as promising directions for future work. Next, we give an overview of our model and our main contributions.

\subsubsection*{Informal Model}
A seller seeks to maximize expected revenue by making sequential take-it-or-leave-it price offers over time. At each step, the seller selects an offer consisting of a buyer and a posted price. The buyer accepts if their private valuation exceeds the posted price, in which case a sale occurs and the process terminates; otherwise, the seller may continue with further offers. The number of offers is limited by a deadline that caps the total number of offers. This deadline may be known in advance or random with a known distribution. The seller knows the distribution of buyer valuations, which may be correlated. 

A result of \citet{xiao2020complexity} already implies that determining an optimal solution to this problem is NP-hard, even when buyer valuations are independent and the horizon is known; in fact, the hardness holds even when each buyer’s valuation has support of size three.

 %The seller needs to decide which
%buyer to approach and what price to post at each step.
%We study optimization under uncertainty in both the buyers' valuations and the
%deadline. Buyers’ valuations are drawn from a known joint distribution that may
%exhibit arbitrary correlation across buyers, and
\subsection{Main Contributions}

 Our main contributions are polynomial-time
approximation algorithms for different regimes of buyer correlation and deadline uncertainty.

\subsubsection{Approximation guarantees.}
In the deterministic, known-horizon model, we obtain a $(1-1/e)$-approximation for arbitrary correlated valuations, and show that this guarantee is tight: no polynomial-time algorithm can achieve a strictly larger approximation factor unless $\mathrm{P}=\mathrm{NP}$. The hardness follows via a reduction from weighted max-$k$-cover.

For general unknown horizons with independent valuations, we present two constant-factor algorithms: a simple deterministic $\tfrac{1}{4}$-approximation and a randomized $(\tfrac{1}{2}-\varepsilon)$-approximation whose runtime scales inversely with $\varepsilon$. For correlated valuations under arbitrary unknown horizons, we obtain an $\Theta(1/\log M)$-approximation, where $M$ is the number of buyers.
Finally, we show that stronger guarantees are possible when the horizon distribution has additional structure. In particular, we obtain a $(1-1/e)$-approximation for independent valuations under a geometric horizon distribution, and a $\tfrac{1}{4}(1-1/e)$-approximation for correlated valuations when the horizon distribution satisfies the Increasing Failure Rate (IFR) property. We summarize the resulting approximation landscape in Tables~\ref{tab:results-core} and~\ref{tab:results-special}.

\subsubsection{Technical ideas.}
We develop our approximation algorithms through two broad approaches.  The first
is to transform the problem into a constrained subset-selection problem over a
ground set of possible offers, where each offer is specified by an agent (buyer), a
price, and its position in the sequence.  The
second is a relax-and-round approach: we formulate a mathematical program that
upper-bounds the optimal expected revenue while capturing the relevant
feasibility constraints, and then round its solution while capturing a constant factor of the relaxation.

If we work directly with the set of all possible offers indexed by agent (buyer), price, and position, the resulting expected-revenue objective is a set function that, in its raw form, appears to lack useful structural properties. Our first step is therefore to impose a simple structural principle that is inherent to optimal solutions in the known-horizon setting: for any feasible set of buyer–price pairs, it is optimal to execute them in decreasing order of price. Under this canonical ordering, for the known-horizon problem with correlated valuations, the induced objective becomes monotone and submodular, and feasibility is captured by a partition matroid constraint. This immediately yields a $(1-1/e)$-approximation (Section~\ref{sec:known}). We also show that this guarantee is tight unless $\mathrm{P}=\mathrm{NP}$.

Beyond yielding the optimal approximation factor, this ordering principle reveals a unifying structure. $(1-1/e)$-approximations were previously known for several special cases of our setting, typically via distinct relax-and-round arguments (see Section~\ref{sec:known}). In contrast, the order-first reduction exposes a common diminishing-returns phenomenon under a natural canonical order, and extends cleanly to general correlated valuations in the known-horizon model.

When the horizon is random, however, there is generally no execution order that can be fixed in advance for all feasible subsets: the optimal order may depend on which offers are selected. This prevents a direct extension of the ordering-based reduction. One important exception is the memoryless geometric horizon, which captures uniform discounting of future revenue—a common alternative objective in practice. Memorylessness restores a set-independent canonical order when valuations are independent, and under this ordering the induced objective is again monotone and submodular. This allows us to extend the $(1-1/e)$ guarantee to the geometric-horizon model (Section~\ref{sec:unknown-geom}).

For general unknown horizons with independent valuations, we pursue two new directions. First, in the spirit of the approach above, we derive additional structural properties of optimal policies and use them to define a more structured objective. This objective is monotone and, while not submodular, admits a \emph{submodular order}. This property enables efficient optimization via a simple and fast directed local-search procedure, yielding a deterministic $\tfrac{1}{4}$-approximation (Section~\ref{sec:unknown-arbitrary-quarter}).

Second, we take an approach inspired by the relax-and-round paradigm, to obtain a stronger approximation. We introduce a mixed-integer linear programming (MILP) relaxation that upper-bounds the optimal expected revenue while explicitly accounting for horizon uncertainty, and show that it can be solved to within a $(1-\varepsilon)$ factor using a PTAS. We then design a randomized transformation that selectively discards offers from the resulting solution in order to produce a sequence of offers whose expected revenue is within a factor $\tfrac{1}{2}$ of the MILP optimum. This transformation relies on a contention-resolution–style argument enabled by the way the relaxation encodes horizon uncertainty and relaxes uncertainty in buyer valuations. Overall, we give a $(\tfrac{1}{2}-\varepsilon)$-approximation under arbitrary unknown horizons (Section~\ref{sec:structural-relaxation}). We also show that no algorithm can achieve a guarantee strictly better than $\tfrac{1}{2}$ when compared with the MILP formulation.

When valuations are correlated and the horizon is unknown, the additional structure exploited in the independent case no longer holds, and direct extensions of the above ideas can be shown to fail. In this regime, we return to the greedy intuition underlying the known-horizon case and investigate how much of it persists under random horizons. When the horizon distribution belongs to an Increasing Failure Rate (IFR) family, we show that there exists a representative deterministic time scale that captures a constant fraction of the optimal expected revenue, enabling a constant-factor approximation via a median-based reduction (Section~\ref{sec:IFR}). For fully arbitrary horizon distributions, we propose an algorithm that selects the best out of many fixed horizon greedy solutions and show that it is exactly $\Theta(1/\log M)$-approximate
(Section~\ref{sec:buckets}).

\begin{table}[h]
\centering
\caption{\textbf{Current approximation landscape.}
Entries marked with $\star$ are obtained in this paper. Note that, 
\citet{brubach2023onlinematchingframeworksstochastic} give a randomized $1/2$-approx.\ for unknown horizon and independent weighted Bernoulli valuations. The $(1-\varepsilon)$ for  known horizon and independent valuations is due to
\citet{segev2020efficient} and \citet{fu2018ptas},\citet{chen2016combinatorial} who give an EPTAS and PTAS respectively.}
\label{tab:results-core}
\begin{tabular}{@{}lcc@{}}
\toprule
& \textbf{Independent valuations} & \textbf{Correlated valuations} \\
\midrule
\textbf{Known horizon}
& $(1-\varepsilon)$
& $(1-1/e)^{\star\ddagger}$ \\
\textbf{Unknown horizon}
& $\tfrac14^{\star}$ (det.) and $(\tfrac12-\varepsilon)^{\star}$ (rand.)
& $\Theta(1/\log M)^{\star}$ \\
\bottomrule
\end{tabular}

\footnotesize
$^{\ddagger}$ We show that $(1-1/e)$ is optimal: no poly-time algorithm can achieve $>(1-1/e)$ unless $\mathrm{P}=\mathrm{NP}$.
\end{table}

\begin{table}[h]
\centering
\caption{\textbf{Improvements under additional horizon structure.}} %Entries marked with $\star$ are obtained in this paper.}
\label{tab:results-special}
\begin{tabular}{@{}lll@{}}
\toprule
\textbf{Horizon structure} & \textbf{Valuation structure} & \textbf{Approx.\ guarantee} \\
\midrule
Geometric (unknown) & Independent & $(1-1/e)$ \\
IFR (unknown) & Correlated & $\tfrac14(1-1/e)$ \\
\bottomrule
\end{tabular}
\end{table}

\section{Related Work}
Our work is connected to several streams across economics, computer science, and
operations research, including bargaining, sequential posted pricing
in mechanism design, and stochastic probing/selection models.

\paragraph{Bargaining}
Sequential take-it-or-leave-it offers are classically studied in one-sided-offer
bargaining models, where outcomes are characterized via equilibrium analysis under
time frictions such as uniform discounting and deadlines
\citep{stahl1972bargaining,rubinstein1982perfect,sobel1983bargaining,
fudenberg1983bargaining,fudenberg1985infinite}.
In contrast, we adopt an algorithmic perspective and study the problem of
approximately maximizing seller revenue under a range
of deadline models beyond uniform discounting and fixed deadlines.
In the single-call model, where each buyer receives at most one offer, truthful
acceptance is optimal: the buyer accepts iff her value is at least the posted
price, so the seller’s problem is purely computational.
In the multi-call model, we restrict to the same threshold-acceptance rule to isolate
the algorithmic structure.

\paragraph{Sequential Posted Pricing and Mechanism Design.}
Sequential posted-pricing mechanisms (SPMs) are a standard tool in algorithmic
mechanism design, where the goal is to approximate the revenue of the \emph{optimal
mechanism} using simple sequential posted prices under assumptions such as
independent private values
 \citep{chawla2010multi}. Related work also develops
approximation schemes for computing near-optimal SPMs in special cases, including
multi-unit settings \citep{chakraborty2010approximation}. In contrast to SPMs, we benchmark
against the \emph{best sequential pricing policy} feasible under an explicit deadline
model, rather than against the revenue of the optimal mechanism. Moreover, whereas existing SPM
guarantees are fundamentally based on independence and do not consider deadlines as a
constraint, our focus includes correlated valuations and a range of
deterministic and stochastic deadline models.

\paragraph{Stochastic Probing and Query-commit Models.}
A closely related algorithmic literature studies stochastic probing and
selection-stopping problems (\citealp{segev2020efficient,fu2018ptas,chen2016combinatorial}), as well as related hiring formulations
(\citealp{epstein2024selection,purohit2019hiring}).
In these models, a decision-maker probes elements subject to feasibility constraints
and commits upon success; this probe--commit structure is equivalent to making
sequential offers and terminating upon the first acceptance, and a hard probe budget
corresponds to a deterministic horizon in our setting. Accordingly, algorithmic
guarantees and approximation schemes from this literature yield guarantees for our
independent-values, deterministic-horizon special cases. In contrast, our setting emphasizes deadline uncertainty through random
horizons and allows arbitrary correlation in buyer valuations.

\emph{Adaptivity Gaps and Submodularity.}
\citep{gupta2017adaptivity} study stochastic probing for monotone submodular objectives under
prefix-closed constraints and show constant adaptivity-gap bounds under independence
assumptions. In our model, adaptivity is not the
relevant issue: non-adaptive sequential policies are without loss of generality
(see Remark \ref{rem:nonadaptivity-wlog}). Moreover, while our known-horizon objective admits a natural
monotone submodular structure, we show that this structure importantly breaks under unknown
(random) horizons, creating challenges outside the scope of these analyses.

\emph{Matching with Unknown Patience.}
For our \emph{unknown-horizon model with independent valuations}, the closest
structural connection is \citet{brubach2023onlinematchingframeworksstochastic}, who
study online matching with stochastic rewards and unknown patience. In their setting, a single online vertex probes incident edges sequentially; probing
an edge corresponds to making an offer in our setting, and their probe--commit rule
matches our termination upon the first acceptance. Their \emph{unknown patience}
parameter, a probe budget that is only revealed when exhausted, corresponds
directly to our random horizon. They focus on an independent stochastic-rewards model in which each edge succeeds
with probability $p$ and yields a fixed weight $w$ upon success, which is a special
case of our independent value model (weighted Bernoulli valuations). They identify
an optimal ordering policy under geometric patience and give a randomized
$1/2$-approximation for arbitrary patience distributions in this rewards model.
In contrast, we study both known and unknown-horizon structure under general independent (not necessarily weighted Bernoullis) as well as correlated
valuations. %and obtain a $(1-1/e)$-approximation under
%a geometric horizon; for arbitrary random horizons we provide a deterministic
%$1/4$-approximation and a randomized $(1/2-\epsilon)$-approximation. %Moreover, we also incorporate valuation correlation and additional deadline models.

\section{Model and Preliminaries}
\label{sec:model}

We refer to buyers as agents throughout. Each agent \(i\in I\) has a private value \(V_i\) drawn from a known distribution, and an offer
of price \(r\in\mathbb{R}_{\ge 0}\) to agent \(i\) is accepted if and only if
\(V_i \ge r\). The process stops at the first accepted offer, or when the
available number of offers (the horizon) is exhausted. Recall that we study the problem along two orthogonal modeling
axes: the dependence structure of agents' values, and the (known or random)
deadline on the number of offers.

\subsubsection*{Value models}
Along the value axis, the two main settings we consider differ in the
dependence structure of agents' values.
In the independent-values model, the values $\{V_i\}_{i\in I}$ are
mutually independent.
In the correlated-values model, the valuation profile
$V=(V_i)_{i\in I}$ is drawn from an arbitrary known joint distribution on
$\mathbb{R}_{\ge 0}^{I}$.
Although our base procedure offers at most one price to each agent, the
correlated-values model also subsumes a \emph{multi-call} variant in which the
same agent may receive multiple offers at different prices (see Section \ref{subsec:multi-call-def}). 
%Lemma~\ref{lem:multicall-special-case}).

In both value models, for each agent $i\in I$ and price $r\in\mathbb{R}_{\ge 0}$,
we denote $p_i(r) \;:=\; \Pr[V_i \ge r]$
for the marginal acceptance probability of offering price $r$ to agent $i$.

\subsubsection*{Horizon models (deadline)}
Independently of the value model, we consider two horizon models.
In the known-horizon setting, the principal may attempt at most $n$
offers, where $n$ is known a priori.
In the unknown-horizon setting, the number of available offers is a
random variable $N$ with known distribution.
In both cases, execution terminates upon the first acceptance or when the
horizon is reached.

Throughout, we assume $V$ and $N$ are mutually independent.

\subsection*{Mathematical Formulation}
From here onward, we use \emph{call} to refer to an ordered pair \(e=(i_e,r_e)\) consisting of an agent
\(i_e\in I\) and a posted price \(r_e\in\mathbb{R}_{\ge 0}\). The universe of
candidate calls is denoted by \(\Calls\). A (nonadaptive) pricing plan is a fixed sequence of calls
\[
  \Plan \;=\; (e_1,e_2,\ldots,e_m),
  \qquad e_t = (i_t,r_t)\in \Calls \;\;\text{for all } t\in[m].
\]

Consider a deterministic horizon \(n\in\mathbb{Z}_{\ge 0}\).
Let \(\mathbf{v}=(v_i)_{i\in I}\in\mathbb{R}_{\ge 0}^{I}\) denote a valuation
realization. Under plan \(\Plan\), at most the first \(\min\{m,n\}\) calls can be
attempted, and execution stops at the first acceptance. The realized revenue is
\[
  \RevV{n}{\Plan}{\mathbf{v}}
  \;=\;
  \sum_{t=1}^{\min(m,n)}
    \underbrace{r_t \cdot \mathbf{1}\{\mathbf{v}_{i_t}\ge r_t\}}_{\text{revenue if call }e_t\text{ accepts}}
    \cdot
    \underbrace{\prod_{s<t} \mathbf{1}\{\mathbf{v}_{i_s}<r_s\}}_{\text{all earlier calls reject}}.
\]
Here \(\mathbf{1}\{A\}\) denotes the indicator of event \(A\), which equals \(1\)
when \(A\) holds and equals \(0\) otherwise.

Let \(V=(V_i)_{i\in I}\) be the random valuation profile. The expected revenue at
horizon \(n\) is
\[
  \RevN{n}{\Plan}
  \;:=\;
  \mathbb{E}_V\!\bigl[\RevV{n}{\Plan}{V}\bigr].
\]
In the unknown--horizon setting, the number of available offers is a random
variable \(N\) with known distribution. The overall expected revenue of plan
\(\Plan\) is
\[
  \RevUnk{\Plan}
  \;:=\;
  \mathbb{E}_N\!\bigl[\RevN{N}{\Plan}\bigr]
  \;=\;
  \mathbb{E}_{V,N}\!\bigl[\RevV{N}{\Plan}{V}\bigr].
\]
Note that \(\RevV{n}{\Plan}{\mathbf{v}}\) is a
realized revenue, \(\RevN{n}{\Plan}\) takes expectation over \(V\) for a fixed
\(n\), and \(\RevUnk{\Plan}\) takes expectation over both \(V\) and \(N\).

We restrict attention to \emph{single-call} plans in which each agent receives at most one call.
Equivalently, a plan \(\Plan=(e_1,\ldots,e_m)\) is feasible if
$
  \bigl|\{\, t\in[m] : i_t = i \,\}\bigr| \le 1
  \text{for all } i\in I.
$
We assume the distributions of \(V\) and \(N\) are known. Our objective is to
choose a feasible pricing plan maximizing expected revenue,
\[
  \max_{\Plan}\;\RevUnk{\Plan}
  \qquad\text{s.t.}\qquad
  \bigl|\{\, t\in[m] : i_t = i \,\}\bigr| \le 1 \;\;\text{for all } i\in I.
\]
We denote the optimal expected revenue by \(OPT\), and let \(\pi^\star\) be an
optimal plan. 

We say an algorithm achieves an \(\alpha\)--approximation if it runs in
time polynomial in the encoding size of the instance (see Appendix \ref{subsec:input-representation-and-computational-model} for more details) and returns a feasible plan
\(\Plan\) satisfying
$\RevUnk{\Plan} \;\ge\; \alpha \cdot OPT.$ In the correlated-valuation setting, we assume access to a value oracle for the revenue objective. Such an oracle could be approximated in practice via Monte Carlo sampling from the joint distribution of valuations, but we ignore this approximation issue and treat the oracle as exact for simplicity. Using standard techniques from the literature, it can be shown that all of our approximation guarantees continue to hold under this approximate-oracle model.

\begin{remark}[Non-adaptivity is without loss]
\label{rem:nonadaptivity-wlog}
 When a call \(e\) is reached, the
only revealed information is that all earlier calls have failed and the horizon
has not yet expired, which is exactly the event under which \(e\) would be
attempted. Consequently, allowing the seller to adapt
future calls based on observed outcomes provides no additional power, and the
space of fixed pricing plans suffices to capture an optimal solution. %optimal sequence of calls is non-adaptive. % No further information is revealed before termination.
\end{remark}

 %In the correlated-values setting, the joint distribution of \(V\) need not be
%given explicitly; instead we assume sampling access to the joint
%distribution of \(V\): we can draw i.i.d.\ samples \(V^{(1)},V^{(2)},\ldots\)
%from the distribution of \(V\) in time polynomial in the input size. Under standard assumptions, this
%implies an \(\varepsilon\)-additive value oracle for objectives of the form
%\(\RevN{n}{\pi}\) by Monte Carlo estimation, since for any fixed plan \(\pi\) the
%random variable \(\RevV{n}{\pi}{V}\) is bounded and can be evaluated in time
%\(O(|\pi|)\) on each sample. Our running-time
%guarantees are therefore with respect to this standard approximate value-oracle model. We formalize the computational model in Appendix \ref{subsec:input-representation-and-computational-model}.

\subsubsection{Multi-call special case}
\label{subsec:multi-call-def}

We say the model is \emph{multi-call} if plans are permitted to include repeated
agent indices, meaning that there may exist distinct times \(s\neq t\) with
\(i_s=i_t\). This corresponds to making multiple offers to the same underlying
agent at different prices. Lemma~\ref{lem:multicall-special-case} (see Appendix \ref{sec:app-model-proofs}) formalizes how
to express this variant within the correlated-values model under the single-call
constraint, i.e. as a special case of our model.

\begin{remark}[Truthful acceptance under multiple offers]
\label{rem:truthful-multicall}
In the multi-call model described above we implicitly assume truthful agent behavior: whenever an
agent \(i\) is offered price \(r\), the acceptance decision is determined
by the valuation threshold rule \(\mathbf{1}\{V_i \ge r\}\), independent of potential future offers.
\end{remark}

\section{General Valuations and Known Horizons}
\label{sec:known}

We begin with the known-horizon model under correlated valuations.  
Recall that, for independent valuations, near-optimal approximation guarantees are known via PTAS type results \citep{segev2020efficient,chen2016combinatorial,fu2018ptas}.
Results based on LP relaxations and rounding
\citep{epstein2024selection,chakraborty2010approximation} achieve a $(1 - 1/e)$ approximation guarantee.
However, these approaches rely critically on independence and do not extend to general
correlated agents.

%"Standard LP relaxations that keep track of marginal probabilities only suffer from large integrality gaps in the correlated setting and PTAS results that rely on dynamic program transitions have to take into account the entire history of calls" - for later

Perhaps a natural attempt is to frame the problem as subset selection over a
time-indexed ground set of possible calls.  
Each element is a triple \((i,r,t)\), specifying an agent \(i\), a posted price
\(r\), and a time \(t\).  
Feasibility is captured by intersection of two partition matroid constraints: at most one call may be made
to any agent, and at most one call may be made at each time.
However, the expected revenue obtained from a selected set of calls, when viewed
as a set function over this time-stamped ground set, lacks the structure needed to apply standard subset-selection techniques.  
In particular, Example~\ref{ex:nonmono-nonsubmod} (Appendix \ref{sec:app-known-examples}) shows that this objective is neither monotone
nor submodular, even under independent valuations.

We reformulate the problem and impose structure on the objective by separating the choice of \emph{which} calls to make from the
choice of \emph{how} to execute them.  
Specifically, we fix an execution order \(\Order\) on calls and, for any feasible set
\(S\), evaluate its expected revenue by executing the calls in \(S\) according
to \(\Order\).
This yields an \emph{order-induced} set function that depends only on the selected
calls and the fixed order, and allows us to study structural properties of the
objective independently of sequential feasibility concerns.

In the known-horizon setting, there is a canonical choice of execution order.
For any feasible set of calls, it is optimal to execute them in decreasing order
of price.  
We show that under this decreasing-price order, the induced objective is
monotone and submodular.
Moreover, the feasibility constraints correspond to a partition matroid.
As a result, the correlated known-horizon problem reduces to monotone submodular
maximization under a matroid constraint, yielding a \((1-1/e)\)-approximation.
We further show that this approximation factor is optimal unless
\(\mathrm{P}=\mathrm{NP}\). We now formalize this, with full details deferred to Appendix~\ref{sec:app-known-proofs}.

\subsubsection*{Order-induced set functions}
\label{subsec:order-induced}

Given a set \(S\subseteq\Calls\) and a total order \(\Order\) on \(\Calls\), let
\(\pi(S,\Order)\) denote the sequence obtained by listing the elements of \(S\)
in increasing order under \(\Order\). When the order is clear from context, we
write \(\pi(S)\).

\begin{definition}[Decreasing-price total order]
\label{def:global-order}
Let \(\orddown\) denote the total order on the call set
\(\Calls=\{(i,r): i\in I,\ r\in M_i\}\) that sorts calls by nonincreasing price,
with an arbitrary but fixed tie-breaking rule. For any subset \(S\subseteq\Calls\),
let $\pi^{\mathrm{d}}(S) \;:=\; \pi(S,\orddown)$
denote the sequence obtained by listing the elements of \(S\) in increasing
order under \(\orddown\) (equivalently, in nonincreasing price order).
\end{definition}

The overall expected revenue under this order is written as:
$
  \RevUnk{\pi^{\mathrm{d}}(S)}
  \;=\;
  \mathbb{E}_{V,N}\!\bigl[\RevV{N}{\pi^{\mathrm{d}}(S)}{V}\bigr].
$

\begin{lemma}[Decreasing-price order is pointwise optimal]
\label{lem:price-order-optimal}
For any finite \(S\subseteq\Calls\) and any valuation realization
\(\mathbf{v}\in\mathbb{R}_{\ge 0}^{I}\),
\[
  \RevV{\infty}{\pi^{\mathrm{d}}(S)}{\mathbf{v}}
  \;=\;
  \max_{\pi\ \text{a permutation of }S}\ \RevV{\infty}{\pi}{\mathbf{v}}.
\]
Here \(\infty\) indicates that the horizon does not truncate the sequence, i.e.,
\(\RevV{\infty}{\pi}{\mathbf{v}}=\RevV{n}{\pi}{\mathbf{v}}\) for any \(n\ge|\pi|\).
\end{lemma}

Further, we show that
for every valuation realization \(\mathbf{v}\), $\RevV{\infty}{\pi^{\mathrm{d}}(S)}{\mathbf{v}}$ is monotone and submodular in the
selected call set \(S\) (see Appendix~\ref{sec:app-known-proofs}). Based on this key property, we establish the reduction in Theorem \ref{thm:known-horizon-reduction} below (proof in  Appendix \ref{sec:app-known-proofs}).

\begin{theorem}[Known-horizon reduction]
\label{thm:known-horizon-reduction}
Consider a deterministic horizon \(n\in\mathbb{N}\). For any valuation profile
distribution, the known-horizon sequential pricing
problem can be formulated as maximizing a monotone submodular set function
subject to a matroid feasibility constraint.
\end{theorem}

\begin{corollary}[Known-horizon approximation]
\label{cor:known-horizon}
In the known-horizon setting of Theorem~\ref{thm:known-horizon-reduction}, the
continuous-greedy algorithm applied to the multilinear extension, followed by
pipage rounding \cite{calinescu2011maximizing}, returns a \((1-1/e)\)-approximate feasible set
\(S_{\mathrm{cont}}\):
\[
  \RevN{n}{\pi^{\mathrm{d}}(S_{\mathrm{cont}})}
  \;\ge\;
  (1-1/e)\,\RevN{n}{\pi^{\mathrm{d}}(S^\star)},
\]
where \(S^\star\) is an optimal feasible call set.
\end{corollary}

\noindent
%The approximation guarantee follows from the continuous-greedy framework of
%~\cite{calinescu2011maximizing} for
%maximizing a monotone submodular function subject to a matroid constraint, under
%the value-oracle model described in
%Section~\ref{sec:model}. 
The next Theorem \ref{thm:hardness-correlated-1-1e} (proof in Appendix \ref{sec:app-known-proofs}) shows that this result is the ``best possible".

\begin{theorem}[Tightness of the \((1-1/e)\) factor under correlated values]
\label{thm:hardness-correlated-1-1e}
Consider \(\varepsilon>0\). Unless \(\mathrm{P}=\mathrm{NP}\), there is no
polynomial-time \((1-1/e+\varepsilon)\)-approximation algorithm for the
known-horizon sequential pricing problem in the correlated-values setting.
\end{theorem}
We prove tightness via a reduction from (weighted) Max-\(k\)-Coverage.
\citet{feige1998threshold} showed that weighted Max-\(k\)-Coverage is NP-hard to approximate
within a factor \(1-1/e+\varepsilon\) for every \(\varepsilon>0\). %The formal proof of \ref{thm:hardness-correlated-1-1e} appears in Appendix \ref{sec:app-known-proofs}.
\section{Independent Valuations and Unknown Horizons}
\label{sec:unknown-intro}

A natural first attempt in the presence of horizon uncertainty is to reuse the
successful order-based perspective from the known-horizon setting. This approach, however, fails even in the simplest of
settings. Consider two offers: one posts a high price but is accepted with low
probability, while the other posts a lower price but is accepted with high probability.
With a known horizon, it is always optimal to make the higher price offer first.
With an unknown horizon, this preference can reverse: making the easy-to-accept
offer earlier can be better, because delaying risks never reaching it at all. 

In fact, a stronger statement holds. When the horizon is unknown, there need not exist any single total order over the ground set that is optimal for all subsets of available calls. As shown in Example~\ref{ex:set-dependent-order} (Appendix~\ref{sec:app-unknown-quarter-examples}), the relative benefit of executing one offer before another depends on which other calls are available later, so no ordering rule can be chosen independently of the selected set.

One alternative is to revert to the time-expanded ground set, in which each decision is indexed by a buyer, price, and time. However, as Example~\ref{ex:nonmono-nonsubmod} (Appendix~\ref{sec:app-known-examples}) demonstrates, this leads to a non-monotone, non-submodular objective, precluding the direct use of classical submodular maximization techniques.

The remainder of this section develops a sequence of algorithmic ideas that progressively recover structure lost due to horizon uncertainty. In Section~\ref{sec:unknown-arbitrary-quarter}, we introduce a pruning operation that transforms the objective on the expanded ground set into a monotone function while simultaneously reducing the feasibility constraints to a \emph{single} partition matroid. Although the resulting function remains non-submodular, it admits a \emph{submodular order} in the reverse direction of execution time. Leveraging existing results for maximizing functions with a submodular order under matroid constraints, this yields a deterministic $1/4$-approximation algorithm for arbitrary horizon distributions. At a high level, this approach generalizes the method we used for known-horizon: we impose structure on the objective to transform it into a form that is easier to optimize.

In Section~\ref{sec:structural-relaxation}, we present an improved randomized algorithm that achieves a
$(1/2-\varepsilon)$ approximation.  This approach departs from submodular-order
methods and instead combines the solution of a relaxation with online contention resolution to better exploit the
structure of independent valuations under horizon uncertainty.

Finally, in Section~\ref{sec:unknown-geom}, we show that when the horizon length is drawn from a geometric random variable, memorylessness induces a
set-independent ordering under which the expected revenue function is monotone
and submodular.  This allows us to match the optimal $(1-1/e)$ approximation
ratio from the known-horizon setting, despite the presence of horizon
uncertainty.

\subsection{General Horizons: A Deterministic Algorithm}
\label{sec:unknown-arbitrary-quarter}

In this section we present a deterministic algorithm for general horizon
distributions. The main technical ingredient is our pruning function,
defined on the time-indexed ground set. This pruning function
restores structural properties absent from the original objective—monotonicity and a submodular order—while
preserving the relevant optimization value. We show that the resulting
objective admits a reverse-time submodular order, which allows us to
instantiate the submodular-order framework of
\citep{pmlr-v202-udwani23a} and obtain a
$1/4$-approximation via a simple deterministic procedure.

Before the
formal analysis, we develop a conceptual overview in stages: we (i) identify the
source of non-monotonicity and non-submodularity (Section \ref{subsec:marginal-future}),
(ii) define a pruned evaluation $G(\cdot)$ on the time-indexed ground set (Section \ref{subsec:prune-def}), (iii) explain how pruning restores a submodular order (Section \ref{subsec:prune-order-intuition}), and (iv) complete the definition of the admitted submodular order (Section \ref{subsec:within-time-tiebreak}).

\subsubsection{Source of non-submodularity}
\label{subsec:marginal-future}

Before presenting our remedy for the non-monotonicity and non-submodularity of
the objective on the expanded ground set, it is helpful to see why these two
failures are closely related.

Fix an execution order and consider a plan written in sequence as \((P,e,S)\),
where \(P\) is the prefix before offer \(e\) and \(S\) is the suffix after \(e\).
Assume that \(e\) is executed at time \(t=\ti{e}\). For any ordered set \(A\),
recall $\RevUnk{A}$ as the expected revenue from executing $A$ and let $Fail(A):=\Pr[\text{no acceptance occurs in }A]$,
where \(Fail(A)\) is taken over the acceptance events (and is independent of the
horizon \(N\)). A standard conditioning argument yields
\[
\RevUnk{P,e,S}
=
\RevUnk{P}
+Fail(P)\Big(q_t\,r_e\,p_e+(1-p_e)\,\RevUnk{S}\Big),
\qquad
\RevUnk{P,S}
=
\RevUnk{P}
+Fail(P)\,\RevUnk{S}.
\]
Therefore the marginal contribution of inserting \(e\) between \(P\) and \(S\) is
\begin{equation}
\label{eq:marginal-prefix-suffix}
\Delta_e(P,S)
:=
\RevUnk{P,e,S}-\RevUnk{P,S}
=
Fail(P)\,p_e\,\bigl(q_t\,r_e-\RevUnk{S}\bigr).
\end{equation}
Equation~\eqref{eq:marginal-prefix-suffix} makes explicit that even with a fixed
prefix \(P\), the marginal gain from adding \(e\) depends on the value of the
suffix through \(\RevUnk{S}\). Suppose there exist two suffix sets \(S\subseteq S^+\)
such that \(\RevUnk{S^+}<\RevUnk{S}\), i.e., adding available future calls lowers
the expected revenue of the suffix plan. Then for any fixed prefix \(P\),
\[
\Delta_e(P,S^+) - \Delta_e(P,S)
=
Fail(P)\,p_e\,\bigl(\RevUnk{S}-\RevUnk{S^+}\bigr)
\;>\;0,
\]
so the marginal gain from inserting \(e\) \emph{increases} when the available
future set grows. This is exactly the opposite of the diminishing-returns
behavior required by submodularity.

Intuitively, a larger set of calls $S^+$ available later should only help overall revenue: the seller
can always refrain from using additional options. The problem is that the term
\(\RevUnk{S}\) in \eqref{eq:marginal-prefix-suffix} does not represent ``the best
use'' of the later calls; it represents the revenue of a fixed execution, and
under that evaluation, adding calls can lower the
revenue. This motivates pruning: we evaluate any set \(S\) by first discarding
the calls that are not worth using and keeping only a best-performing choice of
at most one call per time. With this evaluation, enlarging \(S\) can never hurt,
because any newly added calls can be pruned away.

\subsubsection{Pruned evaluation on the time-indexed ground set}
\label{subsec:prune-def}

We now define the \emph{pruning function} \(G(\cdot)\) on the time-indexed ground set $\CallsT=\{(i,r,t):i\in I,\,r\in M_i,\,t\in[T]\},$
and interpret it as the optimal revenue obtainable by selecting at most one offer per time.
Define the family of \emph{time-feasible} sets by
\[
\mathcal I_{\mathrm{time}}
:= \Bigl\{
A \subseteq \CallsT :
\forall\,t\in[T],\ \bigl|\{e\in A: \ti{e}=t\}\bigr|\le 1
\Bigr\}.
\]
For any $A\subseteq\CallsT$, let $\pi(A)$ denote the plan that executes the calls in $A$ in increasing designated time, and define
\[
G(S)
\;:=\;
\max\bigl\{\RevUnk{\pi(A)}:\ A\subseteq S,\ A\in \mathcal I_{\mathrm{time}}\bigr\}.
\]

This definition highlights two roles played by pruning: (i) it replaces raw
suffix revenue by an internally computed future value that is monotone in the
available set \(S\), and (ii) it enforces the one-offer-per-time restriction
within the definition of the objective. The remaining feasibility constraint in
our model is the agent-side call restriction, which we enforce explicitly
via a partition matroid.

We can efficiently compute \(G(\cdot)\) via a backward value-to-go recursion.
At time \(t\), among the calls \(e\in S_t:=\{e\in S:\ti{e}=t\}\), the recursion compares offers via the
\emph{one-step score}
\begin{equation}
\label{eq:one-step-score}
q_t\,\pr{e}\,p_e \;+\; (1-p_e)\,\mathrm{(future\ value)} ,
\end{equation}
i.e., immediate expected revenue if the horizon reaches time \(t\), plus the future value if the offer is rejected.
Lemma~\ref{lem:envelope-g} (see Appendix \ref{sec:app-unknown-quarter-proofs} for proof) formalizes this computation and shows that it attains the definition of \(G(S)\).

\begin{lemma}[$G$ computed by a backward recursion]
\label{lem:envelope-g}
For each \(S\subseteq\CallsT\), define \(G_{T+1}(S):=0\), and for \(t=T,T-1,\ldots,1\) define
\[
G_t(S)
:=
\max\!\Bigl\{
  0,\;
  \max_{e\in S_t}
     \bigl(
        q_t\,\pr{e}\,p_e
        + (1-p_e)\,G_{t+1}(S)
     \bigr)
\Bigr\},
\qquad
S_t:=\{e\in S:\ti{e}=t\}.
\]
Then \(G(S)=G_1(S)\), where \(G(S):=\max\{\RevUnk{\pi(A)}:\ A\subseteq S,\ A\in \mathcal I_{\mathrm{time}}\}\).
\end{lemma}

\subsubsection{How pruning restores submodular order}
\label{subsec:prune-order-intuition}

We now explain how the pruning function \(G(\cdot)\) remedies the failures of
monotonicity and diminishing returns identified earlier, and how this structure
leads naturally to a submodular-order formulation.

We have already seen that pruning fixes the monotonicity issue. %when the set of
%calls available grows, the pruned evaluation cannot
%become worse, because newly added calls can be ignored. We now turn to the question of submodularity, and 
We start by discussing why pruning does \emph{not}
 restore full submodularity but only a weaker version of it.

Recall the execution of ordered plan \((P,e,S)\) from \eqref{eq:marginal-prefix-suffix},
where \(P\) is the prefix before \(e\) and \(S\) is the suffix after \(e\). When
we enlarge the set of calls, the enlargement can occur either in the
prefix (adding more calls that may appear before \(e\)) or in the suffix (adding
more calls that may appear after \(e\)).

If the enlargement occurs in the prefix, pruning provides no useful control: it may keep some prefix calls and
discard others, changing the factor \(Fail(P)\) in
\eqref{eq:marginal-prefix-suffix}. Since \(Fail(P)\) can increase or decrease as the available prefix grows, we cannot expect a 
diminishing-returns property in this direction, and therefore cannot expect
full submodularity.

In contrast, If we enlarge only the suffix, \(Fail(P)\) in
\eqref{eq:marginal-prefix-suffix} is fixed. Then the marginal gain varies only
through the suffix term \(\RevUnk{S}\), which pruning evaluates monotonically in
the available set of suffix calls $S$. Hence the marginal benefit of adding \(e\) weakly
decreases as the $S$ grows.

This is the pattern we exploit: we obtain diminishing returns only
when we compare marginal gains of elements as the available suffix is
enlarged. Accordingly, the structure is captured by a
\emph{submodular order} that respects reverse time index: it compares marginal gains
only under enlargements that occur strictly \emph{after} (larger $t$) the elements already
present. We recall the formal definition.

\begin{definition}[Submodular order {\citep[Eq.~(1)]{pmlr-v202-udwani23a}}]
\label{def:submodular-order}
Let $f:2^N\to\mathbb{R}$ and let $\pi$ be a permutation of $N$.
For any set $S\subseteq N$, define
$r_\pi(S):=\max\{\pi(i):i\in S\}$, and for any set $C\subseteq N$ define
$\ell_\pi(C):=\min\{\pi(i):i\in C\}$.
The permutation $\pi$ is a \emph{submodular order} for $f$ if for all
$B\subseteq A\subseteq N$ and all $C\subseteq N$ satisfying
$\ell_\pi(C)>r_\pi(A)$,
\[
  f(C\mid A)\;\le\; f(C\mid B),
\]
where $f(C\mid S):=f(S\cup C)-f(S)$.
\end{definition}

Ordering elements by reverse time is the candidate submodular order for
\(G\) across different times. To obtain a total order on the
ground set \(\CallsT\), however, we must also specify how to order elements that
share the same time index. This is essential because \(G\) enforces the
one-offer-per-time restriction internally: in the recursion, \(G_t(S)\) selects
at most one call from the the set of calls available at $t$. We discuss this in more detail next.

\subsubsection{Submodular order within a time period}
\label{subsec:within-time-tiebreak}

The effect of adding an element
\(e=(i,r,t)\) is governed by how it compares to the best competing offer
already present at time $t$.

Recall that this within-time comparison is determined by the one-step score: 
$
q_t\,\pr{e}\,p_e \;+\; (1-p_e)\,G_{t+1}(S)
$ of $e$ with the corresponding score for the best offer already in $S_t$. 
The first term is the expected revenue if \(e\) is accepted at time \(t\); the
second term is what can still be earned if \(e\) is rejected and the process
reaches later times. As the suffix grows, \(G_{t+1}(S)\) increases,
and the best offer in \(S_t\) may shift. The diminishing-returns effect comes
from what \(e\) replaces at time \(t\): with a small suffix, the current best
option in \(S_t\) is more conservative (higher acceptance probability), whereas with a
larger suffix the current best option can be riskier (lower acceptance probability). Thus, if \(e\) is selected in both cases, it replaces a
weaker baseline under the small suffix, so the improvement from adding \(e\) is larger for the small suffix. Therefore, within a fixed time
period, elements admit a submodular order given by increasing acceptance
probability.

We are now ready to formally define the complete submodular order for $G$.

\begin{definition}[Reverse-time order]
\label{def:sigma-rev}
Let $\sigma^{\mathrm{rev}}$ denote the permutation of $\CallsT$ that sorts elements by decreasing time index,
breaking ties within each time period by sorting in order of \emph{increasing} acceptance probability $p_e$. 
\end{definition}

\subsubsection{Analysis}
\label{subsec:pruning-analysis}
We show that the pruned evaluation $G$ admits diminishing returns under the
reverse-time order $\sigma^{\mathrm{rev}}$ and is subadditive.  These two facts
allow us to apply the directed local-search framework of \citet{pmlr-v202-udwani23a}, giving
Algorithm~\ref{alg:single-unknown-local-one-set} and the $1/4$-approximation in
Theorem~\ref{thm:single-unknown-quarter}.  Details are deferred to
Appendices~\ref{sec:app-unknown-quarter-proofs} and~\ref{sec:app-unknown-quarter-alg}.

\begin{theorem}[Directed Local Search: $1/4$-approximation]
\label{thm:single-unknown-quarter}
Let $S_{\mathrm{alg}}$ be the set returned by Algorithm~\ref{alg:single-unknown-local-one-set},
and let $OPT:=\max_{\Plan}\RevUnk{\Plan}$ denote the optimal expected revenue
over all feasible plans. Then
\[
   \RevUnk{\pi(S_{\mathrm{alg}})}\;\ge\;\tfrac14\,OPT.
\]
\end{theorem}

\begin{remark}
\label{rem:one-set-local-search}
Although the analysis of Algorithm~\ref{alg:single-unknown-local-one-set} is
carried out in terms of the pruned evaluator $G$, the algorithm need never
evaluate $G(\cdot)$.  Instead, it simply
maintains a feasible set that is already pruned and evaluates its expected
revenue.
\end{remark}

\subsection{General Horizons: A Randomized Algorithm}
\label{sec:structural-relaxation}

In this section, we design a randomized policy for the independent-valuations
model under an arbitrary random horizon.  For any accuracy parameter
\(\varepsilon\in(0,1)\), the policy achieves a \(\tfrac12(1-\varepsilon)\)
approximation.  We describe the algorithm and outline its analysis in three stages.

\paragraph{Stage 1: A mixed-integer upper bound. (Section \ref{sec:MIP-upper-bound})}
We upper-bound the optimal expected revenue using a mixed-integer program with
two sets of decision variables.
The first set are binary variables \(x_e\) indicating whether a time-stamped
call \(e\in\CallsT\) is selected; the feasibility constraints on \(x\) (a
matroid-intersection structure) ensure that any integral solution
\((x_e)_{e\in\CallsT}\) corresponds to a valid call sequence (at most one call
per agent and per time).
The second set are continuous variables \(y_e\ge 0\), where \(y_e\) is
interpreted as the probability that call \(e\) is the successful call.  These variables allow us to express expected revenue as a
linear function of \(y\).

Specifically, we impose only the aggregate constraint \(\sum_e y_e \le 1\) corresponding to at most one successful call, and,
for each call \(e\), the cap on its success probability \(y_e \le p_e x_e\); the program therefore relaxes additional sequential dependencies,
yielding an upper bound \(\OPT_{\mathrm{MILP}}\ge \OPT\)
(Lemma~\ref{lem:time-milp-upper}). Our goal will be to design a
policy whose expected revenue is a constant fraction of \(\OPT_{\mathrm{MILP}}\).

\paragraph{Stage 2: Approximately solving the upper bound (Section \ref{sec:solving-relaxation})}
Rather than solving an LP relaxation, then rounding under matroid-intersection
and budget-type constraints, we give a direct polynomial-time procedure to
compute a near-optimal solution to the mixed-integer upper bound from Stage~1.
The starting point is a knapsack-style structure: once the binary variables
\(x\) are fixed, the remaining optimization over the continuous variables \(y\)
is a fractional knapsack linear program, which is solved greedily by assigning
mass to the selected calls in decreasing density order.
We then show that if the minimum-density call that receives positive mass in an
optimal solution were known, the mixed-integer program can be reduced to a
subproblem that admits an EPTAS.
This leads to a guess-and-solve method: we try polynomially many candidates for
this breakpoint density and keep the best resulting solution, obtaining a
\((1-\varepsilon)\)-approximation to \(\OPT_{\mathrm{MILP}}\).

Algorithm~\ref{alg:time-lowdensity} implements this reduction and, under a
correct guess, returns a near optimal solution to the MILP
(Lemma~\ref{lem:time-alg-ptas}).  Our overall routine Algorithm \ref{alg:relax-attenuate} tries all polynomially
many guesses and selects the best solution.
\paragraph{Stage 3: OCRS-style attenuation (Section \ref{sec:attenuation})}
Stage~2 computes a near-optimal solution \((x,y)\) to the mixed-integer upper
bound. While this solution respects all constraints of the MILP, its objective value reflects an optimistic aggregation of success probabilities that does not fully account for the sequential nature of the execution process. In particular, the relaxation allows success probability mass to be assigned to calls without enforcing that later calls are reached only if earlier calls fail.  For instance, consider two selected calls \(e_1\) then \(e_2\), each of
which is accepted with probability \(1/2\) when offered.  The program may set
\(y_{e_1}=1/2\) and \(y_{e_2}=1/2\).  But in any actual execution,
\(\Pr[e_2\text{ is successful}]=\Pr[e_1\text{ rejects}]\cdot \Pr[e_2\text{ accepts}\mid e_2\text{ is offered}]
= (1-1/2)\cdot (1/2)=1/4\),
so \(e_2\) cannot have success probability \(1/2\).
Thus, while the MILP solution is feasible and meaningful as an upper bound, its objective value cannot in general be realized by directly executing the selected calls in sequence. To bridge this gap, we apply an additional \emph{attenuation} step that modifies the execution of the selected calls—by selectively discarding some of them in a randomized manner—so as to restore consistency with the underlying sequential process. This attenuation preserves a constant fraction of the MILP objective value while producing a valid sequential policy.

A key feature of our formulation is that horizon uncertainty enters only
through the objective coefficients, while the constraints on \((x,y)\) describe
a purely combinatorial call set together with target success probabilities.
This separation enables a contention-resolution--style attenuation argument.
We process calls in increasing time and randomize which selected calls to
attempt so as to preserve a constant fraction of each call's target success
probability.
Concretely, \textsc{Attenuated-Execute} (Algorithm~\ref{alg:half-attenuated})
guarantees that each selected call \(e\) becomes the successful call with
probability at least \(\tfrac12\,y_e\)
(Lemma~\ref{lem:half-attenuated-factored}), yielding in expectation a
\(\tfrac12\)-fraction of the relaxation objective and enabling a direct
comparison to \(\OPT_{\mathrm{MILP}}\).

\subsubsection{The Mixed-integer Upper Bound}
\label{sec:MIP-upper-bound}

%As stated, we now formulate a mixed-integer program that upper-bounds $\OPT$.
Our formulation relaxes the \emph{first-success distribution} induced by any
sequential plan: it introduces variables that allocate probability mass to the
event that a given call is the successful call, and it enforces only the
necessary aggregate constraints.  The program does not encode
\emph{reachability}---that is, the fact that a call can be attempted only if all
earlier calls fail---and therefore discards the sequential coupling between
success events.

Formally, each element $e=(i,r,t)\in\CallsT$ represents offering agent $i$ the
price $r$ at time~$t$.  Let $p_e := \Pr[V_i \ge r]$ denote the acceptance
probability of call $e$, and let $q_t$ denote the probability that the horizon
survives to time $t$ (so $q_{\ti{e}}$ is the survival weight for call $e$).

Feasible call selections must choose at most one call per agent and at most one
call per time step; we encode these requirements via the partition families
\[
\mathcal I_1 := \Bigl\{
S \subseteq \CallsT : 
\sum_{r,t} \mathbf 1\{(i,r,t)\in S\} \le 1\ \ \forall i
\Bigr\}
\qquad
\mathcal I_2 := \Bigl\{
S \subseteq \CallsT :
\sum_{i,r} \mathbf 1\{(i,r,t)\in S\} \le 1\ \ \forall t
\Bigr\}.
\]
We introduce variables $x_e\in\{0,1\}$ and $y_e\ge 0$ for each $e\in\CallsT$,
where $x_e$ indicates whether call $e$ is selected and $y_e$ is the probability
mass assigned to the event that $e$ is the successful call.  The mass variables
satisfy $\sum_{e\in\CallsT} y_e \le 1$ and $y_e \le p_e x_e$ for all $e$.

We formulate the relaxation as
\begin{equation}
\label{eq:time-milp}
\max_{x,y}
\Bigl\{
  \sum_{e\in\CallsT} q_{\ti{e}}\,r_e\,y_e :
  \sum_{e\in\CallsT} y_e \le 1,\;
  0 \le y_e \le p_e\,x_e,\;
  x \in \mathcal I_1 \cap \mathcal I_2
\Bigr\}.
\end{equation}

We note that in contrast to relaxations that incorporate horizon uncertainty
directly into the constraint system by encoding reachability under the random
horizon (e.g., \cite{brubach2023onlinematchingframeworksstochastic}),
the uncertainty in the deadline enters \eqref{eq:time-milp}
only through the objective coefficients \(q_{\ti{e}}\).

Lemma \ref{lem:time-milp-upper} (Proof in Appendix \ref{sec:app-unknown-half-proofs}) states that this relaxation is indeed valid: every feasible
policy induces a feasible solution to the relaxation with the same expected
revenue.

\subsubsection{Solving the Relaxation}
\label{sec:solving-relaxation}
Next, we study the structure of optimal solutions to~\eqref{eq:time-milp}.
Fixing the selection variables \(x\), the remaining optimization over
success mass \(y\) depends on each call only through its time-discounted
price \(q_{\ti{e}}r_e\), which we call its \emph{density}.  Through lemma \ref{lem:breakpoint} (Proof in Appendix \ref{sec:app-unknown-half-proofs}), we show that optimal
solutions have a threshold form in this density order: all positive mass is
assigned to calls above a common cutoff, with at most one call used partially.

\begin{definition}[Density and breakpoint element]
\label{def:time-density}
For each \(e=(i,r,t)\in\CallsT\), define its \emph{density}
$
c_e := q_{\ti{e}}\,r_e .
$
Given a feasible solution \((x,y)\) to~\eqref{eq:time-milp}, denote its
effective support by
$
supp(y) := \{\, e\in\CallsT : y_e>0 \,\}.
$
A call \(g\in supp(y)\) is a \emph{breakpoint element} of \((x,y)\) if
\[
c_g \;=\; \min\{\, c_e : e\in supp(y)\,\}.
\]
We let \(\tau:=c_g\) and call \(\tau\) the \emph{breakpoint density}.
\end{definition}

Before we state and prove lemma \ref{lem:breakpoint}, for technical convenience, we augment \(\CallsT\) with a dummy agent
\(i_\dagger\notin I\) and a terminal dummy time \(t_\dagger\) satisfying
\(q_{t_\dagger}=0\), and include the dummy call
\(e_\dagger:=(i_\dagger,0,t_\dagger)\) with \(p_{e_\dagger}=1\).
This element never conflicts with real calls (it uses a dummy agent and time),
and contributes zero to the objective.  It is used only to ensure that we may
assume that the mass constraint \(\sum_e y_e\le 1\) is tight in an optimal solution. We also assume, that all calls have distinct densities. If ties occur,
we can break them by an arbitrarily small perturbation of posted prices, which
induces a strict total order of densities and changes the objective by an arbitrarily small
additive amount; we therefore restrict attention to the tie-free case.

\begin{lemma}[Breakpoint density-threshold structure]
\label{lem:breakpoint}
In the augmented instance above, there exists an optimal solution
\((x^\star,y^\star)\) to~\eqref{eq:time-milp} with the following properties.

\begin{enumerate}[label=(\roman*),leftmargin=2.5em]
\item {Mass constraint is tight:}
\(\sum_{e\in\CallsT} y^\star_e = 1\).

\item {Density-threshold form:}
Let \(S^\star:=\{e:x^\star_e=1\}\) and order its elements so that
\(c_{e_1}\ge c_{e_2}\ge\cdots\ge c_{e_m}\).
Then there exists an index \(k\in[m]\) such that
\[
y^\star_{e_j}=p_{e_j}x^\star_{e_j}\quad\text{for all }j<k,\qquad
0<y^\star_{e_k}\le p_{e_k}x^\star_{e_k},\qquad
y^\star_{e_j}=0\quad\text{for all }j>k.
\]
Equivalently, letting \(\tau:=c_{e_k}\),
all selected calls with \(c_e>\tau\) are saturated (\(y^\star_e=p_ex^\star_e\)),
all selected calls with \(c_e<\tau\) receive zero mass, and the remaining mass
\(1-\sum_{e\in S^\star:\,c_e>\tau}p_e x^\star_e\) is placed on calls with
\(c_e=\tau\) (in particular, at most one call is partially used if ties are
broken consistently).
\end{enumerate}
\end{lemma}

Lemma~\ref{lem:breakpoint} suggests a natural solution strategy to solve \ref{eq:time-milp}: if we knew the
breakpoint density \(\tau\) (equivalently, a breakpoint element \(g\) with
\(\tau=c_g\)), then all calls with density above \(\tau\) are either fully used or irrelevant.

We now introduce the first component of our final algorithm,
Algorithm~\ref{alg:time-lowdensity}, which computes a near-optimal solution to
the relaxation~\eqref{eq:time-milp} under a guessed breakpoint element \(g\).
Given \(g\) (and \(\tau:=c_g\)), the algorithm restricts attention to calls with
\(c_e\ge \tau\), selects an above-threshold feasible set via a budgeted matroid
intersection subroutine, and outputs a feasible pair \((x,y)\) by saturating
the selected calls and placing any remaining mass on \(g\).

\begin{algorithm}[t]
\caption{\textsc{Breakpoint-Guess-MILP}\((g,\varepsilon)\)}
\label{alg:time-lowdensity}
\DontPrintSemicolon
\KwIn{A guessed call \(g\in\CallsT\); accuracy parameter \(\varepsilon\in(0,1)\)}
\KwOut{A feasible pair \((x,y)\) for~\eqref{eq:time-milp}}

Set \(\tau \gets c_g\).\;

\tcp{\textbf{Contract the guessed breakpoint element}}

Let \(\widehat{\CallsT} \gets \{e\in\CallsT : \ag{e}\neq \ag{g},\ \ti{e}\neq \ti{g}\}\).\;
Let \(\widehat{\mathcal I}_1,\widehat{\mathcal I}_2\) be
\(\mathcal I_1,\mathcal I_2\) restricted to \(\widehat{\CallsT}\).\;

\tcp{\textbf{Discard sub-threshold calls}}

Let \(E_{\ge\tau}\gets \{e\in\widehat{\CallsT} : c_e\ge\tau\}\).\;

\tcp{\textbf{Budgeted matroid intersection on the contracted instance}}

For each \(e\in E_{\ge\tau}\) set \(w_e \gets p_e(c_e-\tau)\) and
\(a_e\gets p_e\).\;
Compute
\[
H^\star \in \arg\max\Bigl\{\sum_{e\in H} w_e:
\ H\in\widehat{\mathcal I}_1\cap\widehat{\mathcal I}_2,\ 
H\subseteq E_{>\tau},\
\sum_{e\in H} a_e \le 1
\Bigr\}.
\]

\tcp{\textbf{Construct the solution \((x,y)\): saturate above-threshold items and put leftover mass on $g$}}

Set \(x_g\gets 1\) and \(x_e\gets \mathbf 1\{e\in H^\star\}\) for \(e\neq g\).\;
Set \(y_e\gets p_e\) for all \(e\in H^\star\), and \(y_e\gets 0\) for all other
\(e\neq g\).\;
Set \(y_g \gets \min\{\,p_g,\ 1-\sum_{e\in H^\star} p_e\,\}\).\;

\Return{\((x,y)\)}\;
\end{algorithm}

We emphasize that the optimization over \(H\) in
Algorithm~\ref{alg:time-lowdensity} is a budgeted matroid intersection problem
and is NP-hard, so \(H^\star\) cannot in general be computed exactly.  We
therefore invoke a \((1-\varepsilon)\)-approximation (EPTAS)
subroutine~\cite{DAKS23}(Prior work established a PTAS for this problem \cite{chekuri2012dependent},\cite{BBGS11}); let \(H^{ALG}\) denote its
output, which we compare against the optimal \(H^\star\) for the contracted
instance.

Next, Lemma \ref{lem:time-alg-ptas} (Proof in Appendix \ref{sec:app-unknown-half-proofs}) show that under a correct guess of the breakpoint element \(g\),
Algorithm~\ref{alg:time-lowdensity} produces a solution within a
factor \((1-\varepsilon)\) of the optimal MILP.

\begin{lemma}[Guarantee under a correct breakpoint guess]
\label{lem:time-alg-ptas}
If Algorithm~\ref{alg:time-lowdensity} is run with a breakpoint element \(g\) of
an optimal solution to~\eqref{eq:time-milp}, then its output satisfies
\[
\sum_{e\in\CallsT} c_e\,y^{ALG}_e
\;\ge\;
(1-\varepsilon)\cdot OPT_{MILP},
\]
where \(OPT_{MILP}\) denotes the optimal value of~\eqref{eq:time-milp}.
\end{lemma}

\subsubsection{OCRS-style attenuation}
\label{sec:attenuation}

Although Stage~2 outputs an integral feasible call set, it does not by itself
specify a policy whose realized first-success
probabilities match the prescribed masses \(y\): in any sequential execution,
whether a call is attempted is coupled with the failure of all earlier calls.
We address this by executing the selected calls in increasing time and applying
an OCRS-style attenuation rule that preserves a constant fraction of the target
first-success mass for each call.

\paragraph{Half-attenuation.}
Fix any feasible \((x,y)\) for~\eqref{eq:time-milp} and let
\(\pi(x)=(e_1,\ldots,e_k)\) be the time-ordered sequence induced by \(x\).  For
each \(t\), define \(Y_{t-1}:=\sum_{s<t} y_{e_s}\) and
\(\alpha_{e_t}:=y_{e_t}/p_{e_t}\) (w.l.o.g.\ assume \(p_e>0\)).  Conditional on
no earlier acceptance, we attempt \(e_t\) with probability
$
  \beta_{e_t}
  \;:=\;
  \frac{\tfrac12\,\alpha_{e_t}}{1-\tfrac12\,Y_{t-1}}.
$
This rule follows the contention-resolution conversions of
\citet{alaei2014bayesian} (``magician's problem''), which achieve a
\(\gamma_k\)-fraction of the prescribed ex-ante mass; for the unit-capacity case
\(k=1\), \(\gamma_1\ge \tfrac12\), giving the constant used here.

We now introduce the second component of our final algorithm,
Algorithm~\ref{alg:half-attenuated} (Full Algorithm appears in Appendix \ref{sec:app-unknown-half-alg}), which implements the half-attenuation rule
above to convert a feasible solution \((x,y)\) into a policy that retains a constant factor of the relaxation objective.

\begin{lemma}[Half-attenuated implementation]
\label{lem:half-attenuated-factored}
Fix any feasible \((x,y)\) for~\eqref{eq:time-milp}, and run
Algorithm~\ref{alg:half-attenuated}.  Then for every \(t\in[k]\),
$
\Pr[\text{$e_t$ is the successful call under }\Plan^{\mathrm{ATT}}]
\;=\;
\tfrac12\,y_{e_t}.
$
Moreover, the attempt probabilities \(\beta_{e_t}\)
lie in \([0,1]\).
\end{lemma}

Our final Algorithm ~\ref{alg:relax-attenuate} is the composition of two routines: our relaxation solver Algorithm~\ref{alg:time-lowdensity} that
computes a near-optimal feasible pair \((x,y)\) for~\eqref{eq:time-milp}, and Algorithm~\ref{alg:half-attenuated} which rounds \((x,y)\). We note that Algorithm \ref{alg:relax-attenuate} is randomized and invokes an EPTAS as a subroutine (via \ref{alg:time-lowdensity}). Consequently, for any fixed
\(\varepsilon>0\), it runs in time polynomial in the input size, up to an
additional multiplicative factor \(f(1/\varepsilon)\), and achieves an
\(\alpha\)-approximation in expectation:
$
  \mathbb{E}\!\left[\RevUnk{\Plan}\right] \;\ge\; \alpha \cdot OPT.
$

\begin{algorithm}[t]
\caption{\textsc{Relax-and-Attenuate}}
\label{alg:relax-attenuate}
\DontPrintSemicolon
\KwIn{Time-indexed call set $\CallsT$ (augmented with $e_\dagger$); parameters $(p_e)_{e\in\CallsT}$, $(q_t)_{t\in[\Tmax]}$; accuracy $\varepsilon\in(0,1)$}
\KwOut{A randomized policy $\Plan^{\mathrm{ALG}}$}

\ForEach{$g\in\CallsT$}{
  $(x^{(g)},y^{(g)}) \gets \textsc{Breakpoint-Guess-MILP}(g)$
  \tcp*{Algorithm~\ref{alg:time-lowdensity}}
}
Let $g^\star \in \arg\max_{g\in\CallsT}\ \Bigl\{\sum_{e\in\CallsT} q_{\ti{e}}\,r_e\,y^{(g)}_e\Bigr\}$.\;
\Return{$\Plan^{\mathrm{ALG}} \gets \textsc{Attenuated-Execute}(x^{(g^\star)},y^{(g^\star)})$}
\tcp*{Algorithm~\ref{alg:half-attenuated}}
\end{algorithm}

Finally Theorem \ref{thm:half-eps-approx} (Proof in Appendix \ref{sec:app-unknown-half-proofs}) states our guarantee.
\begin{theorem}[$\tfrac12(1-\varepsilon)$-approximation]
\label{thm:half-eps-approx}
Let \(\Plan^{\mathrm{ALG}}\) be the randomized policy returned by
Algorithm~\ref{alg:relax-attenuate}. Then
\[
\mathbb{E}\!\left[\RevUnk{\Plan^{\mathrm{ALG}}}\right]
\;\ge\;
\tfrac12(1-\varepsilon)\,\OPT_{\mathrm{MILP}}
\;\ge\;
\tfrac12(1-\varepsilon)\,\OPT.
\]
\end{theorem}

\begin{remark}[Tightness of Analysis]
\label{rem:atten-tight}
Example~\ref{ex:milp-stoch-gap-two} shows that the mixed-integer relaxation
\eqref{eq:time-milp} can overestimate the true optimal expected revenue by a
factor approaching $2$.
Consequently, any rounding procedure that guarantees only a constant fraction of
the MILP objective cannot, in general, achieve a better than $1/2$ approximation
with respect to this benchmark.
\end{remark}

\subsection{Geometric Horizons}
\label{sec:unknown-geom}
A standard way to model time frictions in sequential bargaining is exponential
discounting (e.g., \citealp{fudenberg1985infinite}), which in our formulation is
equivalently captured by a geometric horizon: after each rejection, the process
survives with probability $\rho\in(0,1)$, so the number of available calls $N$ is
geometric.

A key benefit of this memoryless structure is that there exists a single
global execution order that is optimal for \emph{every} feasible subset of
calls.  We adopt the resulting \emph{score order} from
\citet[Theorem~5]{brubach2023onlinematchingframeworksstochastic}.  For each call
$e=(i_e,r_e)\in\Calls$ with success probability $p_e=\Pr[V_{i_e}\ge r_e]$, define
the geometric-horizon score
\[
   \mathrm{Score}_\rho(e)
   := \frac{r_e\,p_e}{\,p_e + (1-p_e)\rho\,}.
\]
The \emph{score order} $\ordscore$ is the total order on $\Calls$ that sorts
calls in nonincreasing $\mathrm{Score}_\rho(e)$, breaking ties deterministically.
For any feasible call set $S\subseteq\Calls$, execution proceeds according to the
restriction of $\ordscore$ to $S$. We note that the validity of this order relies on independence of successes across calls.
Fixing this order induces a clean set-function view: define $F(S)\;:=\;\RevUnk{\Plan(\ordscore,S)}.$
Unlike the deterministic-horizon case, the resulting monotonicity and
diminishing-returns properties need not hold pointwise for each realization of
the valuations and the random horizon.  Nevertheless, we show that $F(\cdot)$ is
monotone and submodular in expectation; the formal lemmas and proofs
appear in Appendix~\ref{sec:app-geom-proofs}. We now state the resulting reduction.

\begin{theorem}[Reduction of the geometric-horizon problem]
\label{thm:geom-reduction}
Under the geometric-horizon single-call model with parameter $\rho\in(0,1)$ and
the global score order $\ordscore$, the optimization problem
\[
   \max_{\substack{S\subseteq\mathcal{C}\\
                   \forall i:\,|\{(i,r)\in S\}|\le 1}}
       \RevUnk{\Plan(\ordscore,S)}
\]
is monotone submodular under a partition matroid constraint; hence continuous
greedy plus pipage rounding yields a $(1-1/e)$-approximation.
\end{theorem}

\section{Correlated Valuations and Unknown Horizons}
\label{sec:IFR}
In this section, we study random-horizon models under \emph{arbitrarily correlated}
valuations.  In the independent-valuation setting, our unknown-horizon algorithms were
driven by additional structure: although the objective is not pointwise submodular, it
admits a reverse-time submodular order, yielding constant-factor approximations even
when the horizon is unknown.  Under arbitrary correlations, this structure collapses.
Example~\ref{ex:submodular-order-failure} shows that reverse time need not be a
submodular order even under a known horizon, and Example~\ref{ex:milp-misleading} shows
that the mixed-integer relaxation used in the independent case can become an arbitrarily
poor upper bound.  Consequently, the two main tools that underpinned our independent
unknown-horizon results do not extend to correlated values.

At this point, our only algorithmic tool that survives correlation is the greedy
algorithm under a known finite horizon, where greedy under the decreasing-price order
achieves a constant-factor approximation.  A natural question is whether one can extend
this guarantee to random horizons by selecting a single deterministic horizon length $t$
(e.g., a quantile of the horizon) and running deterministic greedy with budget $t$.
Example~\ref{ex:no-fixed-quantile} rules out such a ``single time-scale'' reduction for
general horizon distributions: for any fixed quantile level $\alpha\in(0,1)$, there are
instances on which greedy optimized at the $\alpha$-quantile horizon achieves an
arbitrarily small fraction of $OPT$.

The situation improves for horizons with increasing failure rate (IFR).  In this case,
the horizon admits a meaningful representative scale---namely, the median---within which
a constant fraction of the optimal expected revenue is concentrated.  This yields a
median-based reduction: optimizing greedy at the median recovers a constant-factor
approximation even under arbitrary correlations (Section~\ref{sec:IFR}).

For fully arbitrary horizon distributions, however, no such representative scale may exist, and an algorithm must hedge across multiple time scales. We do so by applying our greedy algorithm, originally designed for the deterministic-horizon setting, to the decreasing-order-induced set function associated with each time (t) in the support of the horizon distribution. We then take the best solution among these deterministic greedy solutions. A key bucketing argument shows that only $O(\log M)$ such time scales are needed to capture a constant fraction of the optimal expected revenue. Consequently, the best deterministic greedy solution achieves a $\Theta(1/\log M)$ -approximation.
(Section~\ref{sec:buckets}).

\subsection{IFR Horizons}
\label{sec:unknown-ifr}

We consider an unknown horizon $N$ drawn from an increasing-failure-rate (IFR)
distribution, and analyze a simple \emph{median-based} reduction to the
deterministic-horizon problem.  Let $m$ be the (left) median of $N$, i.e.,
$\Pr[N\ge m]\ge\tfrac12$ and $\Pr[N\ge m{+}1]\le\tfrac12$.
We run the deterministic-horizon $(1-1/e)$-approximation algorithm with budget
$m$ (Corollary~\ref{cor:known-horizon}), obtaining a feasible set $S^{\mathrm{med}}$,
and define the induced plan $\Plan^{\mathrm{med}}:=\pi^{\mathrm{d}}(S^{\mathrm{med}})$.

\begin{theorem}[Median-based algorithm for IFR horizons]
\label{thm:ifr-median-greedy}
The deterministic algorithm applied at the median $m$ achieves a
$\tfrac14(1-\tfrac1e)$-approximation for the stochastic-horizon problem when $N$
is IFR:
\[
  \RevUnk{\Plan^{\mathrm{med}}}
  \;\ge\;
  \tfrac14\Bigl(1-\tfrac1e\Bigr)\,OPT^{\mathrm{IFR}}.
\]
\end{theorem}

The proof of Theorem~\ref{thm:ifr-median-greedy} appears in
Appendix~\ref{sec:app-ifr} and rests on two ingredients.  First, the optimal
deterministic-horizon revenue exhibits diminishing returns, implying a
\emph{scaling} property (Lemma~\ref{lem:g-scaling}): extending the horizon by a
factor $k$ increases the optimum by at most a factor $k$.  Second, IFR implies a
\emph{block tail bound} (Lemma~\ref{lem:ifr-blocks}): the survival probability
across successive blocks of length $m$ decays geometrically, so the expected
number of $m$-blocks is $O(1)$.  Together, these show that revenue obtainable
beyond the median contributes only a constant-factor fraction of
$OPT^{\mathrm{IFR}}$, and hence the median budget captures a constant fraction of
the stochastic-horizon optimum.
\subsection{General Horizons}
\label{sec:buckets}

We now give an approximation algorithm for an arbitrary horizon distribution, specified by survival probabilities $(q_t)_{t=1}^T$. Algorithm ~\ref{alg:best-greedy-horizon} is simple: for every deterministic horizon length $h$ in the support of the survival distribution, it runs the deterministic-horizon greedy algorithm, evaluates the resulting plan under the unknown-horizon objective, and returns the best such plan. We show that the resulting approximation guarantee incurs a tight logarithmic loss in the number of agents. In this section, we prove the positive side: the algorithm obtains an
$\Omega(1/\log M)$ fraction of the optimal unknown-horizon revenue. The
matching lower bound, showing that this logarithmic loss is unavoidable for
this algorithm, is deferred to Appendix~\ref{sec:app-ifr}. The loss is driven by a structural property of the optimum: under fully arbitrary horizon distributions, the optimal expected revenue may be spread across multiple time scales, but a bucketing argument shows that $O(\log M)$ such scales suffice to capture a constant fraction of $OPT$.

For the proof of the logarithmic approximation guarantee, we introduce a
bucketing of time periods by survival probability. This bucketing is used only
in the analysis; the algorithm itself simply tries all deterministic greedy
solutions and returns the one with largest unknown-horizon revenue.

Formally, define
\[
  B_k
  :=
  \{\, t\in[T] : 2^{-(k+1)} < q_t \le 2^{-k} \,\},
  \qquad k=0,1,2,\ldots .
\]
Thus, within each bucket $B_k$, survival probabilities vary by at most a
factor of two. Let $n_k:=|B_k|$. Up to this factor of two, the contribution
from bucket $k$ behaves like a deterministic-horizon objective of length
$n_k$, scaled by $2^{-k}$. This allows us to compare the best contribution
obtainable from bucket $k$ to the value obtained by running the deterministic
greedy routine with horizon length $n_k$.

Small-survival buckets cannot be discarded a priori: even late segments of the
horizon may generate substantial expected revenue if they contain many call
opportunities. Instead, the analysis decomposes $OPT$ into bucket
contributions. We show that only $O(\log M)$ buckets can be \emph{heavy} and
carry a significant fraction of the optimal revenue, while the remaining
\emph{light} buckets contribute only a geometrically decaying tail
(Lemma~\ref{lem:nonignorable-count}). We then show that, for every bucket $k$,
the deterministic greedy solution with horizon length $n_k$ obtains a
constant-factor approximation to the best value attainable from that bucket
(Lemma~\ref{lem:bucket-det}). Since Algorithm~\ref{alg:best-greedy-horizon}
tries the deterministic greedy solution for every horizon length, it is at least
as good as the best of these bucket-induced greedy candidates. Combining these
facts gives the claimed $\Omega(1/\log M)$ approximation.

Let $\pi^{\mathrm{BG}}$ be the plan returned by
Algorithm~\ref{alg:best-greedy-horizon}. The full algorithm appears in
Appendix~\ref{sec:app-ifr}.

\begin{theorem}[Best deterministic greedy has tight logarithmic loss]
\label{thm:best-greedy-finite}
The worst-case approximation loss of Algorithm~\ref{alg:best-greedy-horizon}
is $\Theta(\log M)$. Equivalently, on every instance,
\[
  \RevUnk{\pi^{\mathrm{BG}}}
  \;\ge\;
  \Omega\!\left(\frac{1}{\log M}\right)\cdot OPT,
\]
and there is a family of instances for which Algorithm~\ref{alg:best-greedy-horizon}
obtains only an
\[
  O\!\left(\frac{1}{\log M}\right)
\]
fraction of the optimal unknown-horizon revenue.
\end{theorem}

Let $\pi^\star$ be an optimal policy and let
\[
  OPT
  :=
  \RevUnk{\pi^\star}
  =
  \sum_{t=1}^{T}
  q_t\Big(\RevN{t}{\pi^\star}-\RevN{t-1}{\pi^\star}\Big),
  \qquad
  \RevN{0}{\pi}:=0.
\]
For each bucket $k$, define its contribution under $\pi^\star$ by
\[
  V_k(\pi^\star)
  :=
  \sum_{t\in B_k}
  q_t\Big(\RevN{t}{\pi^\star}-\RevN{t-1}{\pi^\star}\Big),
  \qquad\text{so that}\qquad
  OPT
  =
  \sum_k V_k(\pi^\star).
\]
Define the best first-bucket value
\[
  V_0^\star
  :=
  \max_{\pi}
  \sum_{t\in B_0}
  q_t\Big(\RevN{t}{\pi}-\RevN{t-1}{\pi}\Big).
\]
We may assume $q_1=1$ without loss of generality, since the horizon can be
normalized to survive to the first offer. Hence $1\in B_0$ and
$B_0\neq\emptyset$.

Fix $\varepsilon\in(0,1)$. We call bucket $k$ \emph{heavy} if
\[
  V_k(\pi^\star)
  \;\ge\;
  (1-\varepsilon)^k V_0^\star,
\]
and \emph{light} otherwise. Let $\mathcal{K}_{\mathrm{heavy}}$ denote the set
of heavy bucket indices. The choice of $\varepsilon$ affects only constants,
so we fix an arbitrary constant $\varepsilon\in(0,1)$ for the remainder of the
analysis.

We first show that the number of buckets that can be large relative to the
first bucket is logarithmic in the maximum plan length $M$.

\begin{lemma}[Number of heavy buckets]
\label{lem:nonignorable-count}
\[
  |\mathcal{K}_{\mathrm{heavy}}|
  \;\le\;
  1+\left\lceil \log_{\,2-2\varepsilon} M \right\rceil
  \;=\;
  O(\log M).
\]
\end{lemma}

Next, we relate each bucket objective to a deterministic-horizon objective.
Within a bucket, the weights $q_t$ are within a factor of two; hence optimizing
the bucket contribution $V_k(\cdot)$ is equivalent up to constants to
optimizing an unweighted deterministic-horizon objective over $n_k=|B_k|$
steps, for which greedy gives a $(1-1/e)$ guarantee.

For every nonempty bucket $B_k$, let
\[
  \pi_k := \textsc{Greedy}(\mathcal{C},n_k),
  \qquad n_k:=|B_k|.
\]

\begin{lemma}[Deterministic approximation within a bucket]
\label{lem:bucket-det}
Let $B_k$ be a nonempty bucket and let $n_k:=|B_k|$. Let
$\pi_k=\textsc{Greedy}(\mathcal{C},n_k)$ be the deterministic greedy plan for
horizon length $n_k$. Then
\[
  V_k(\pi_k)
  \;\ge\;
  \frac12\Bigl(1-\frac1e\Bigr)\cdot \max_{\pi} V_k(\pi).
\]
\end{lemma}

Finally, the logarithmic dependence in
Theorem~\ref{thm:best-greedy-finite} is tight for
Algorithm~\ref{alg:best-greedy-horizon}. In
Example~\ref{ex:best-greedy-tight}, we construct a family of instances for
which every deterministic greedy solution has only constant unknown-horizon
revenue, while the optimal unknown-horizon policy obtains $\Theta(\log M)$
revenue. Hence the best deterministic greedy solution obtains only an
$O(1/\log M)$ fraction of optimum on these instances.

\section{Final Remarks and Open Questions}
In summary, this paper develops algorithmic tools that expose and
exploit structure in sequential posted pricing with deadlines across several
regimes.  The most compelling open question is whether one can obtain a
constant-factor approximation for arbitrary random horizons under correlated
valuations.  Our results show that techniques that succeed under independence do
not directly extend, but they do not rule out a different approach in this
general regime.

A second open question is to identify which forms of correlation arise
naturally and which of them admit efficient approximations.  In particular,
when correlation is induced solely by multiple calls to the same buyer, so
valuations are independent across buyers but a buyer’s accept and reject
outcomes across prices are coupled through a single latent value, does the
problem admit a constant-factor approximation under arbitrary time horizons?
More broadly, can one show that this within-buyer correlation captures the main
algorithmic difficulty?  Understanding which correlations are fundamentally
hard would sharpen the boundary between tractable and intractable sequential
pricing under deadlines.

\ACKNOWLEDGMENT{
The authors thank Balasubramanian Sivan, Steven Delong, and Benjamin Miller for helpful discussions on the motivation of this work.
}

\bibliography{References}

\section{Appendix - Section \ref{sec:model}}
\label{sec:app-model}
This appendix contains a formal proof for the results in Section~\ref{sec:model} as well as the formal computational model for the problem.

\refstepcounter{subsection}
\subsection*{Formal Proofs}
\label{sec:app-model-proofs}

\begin{lemma}[Multi-call as a special case of correlated values]
\label{lem:multicall-special-case}
The multi-call model is a special case of the correlated-values model with at
most one call per agent.
\end{lemma}

\begin{proof}[Proof of Lemma \ref{lem:multicall-special-case}]
Consider any multi-call instance with agent set \(I\) and candidate prices
\(M_i\) for each agent \(i\). Construct a new instance with expanded agent set
\[
  I' \;:=\; \{(i,r): i\in I,\ r\in M_i\},
\]
and call ground set
\[
  \Calls' \;:=\; \{\, ((i,r),r) : (i,r)\in I' \,\}.
\]
Thus, each agent in \(I'\) appears in exactly one call.

Define a joint distribution over \(V'=(V'_{(i,r)})_{(i,r)\in I'}\) by sampling the
original valuation profile \(V\) and setting \(V'_{(i,r)}:=V_i\) for all
\((i,r)\in I'\). Then, for every \((i,r)\), the acceptance event of the original
call \((i,r)\) coincides with the acceptance event of the corresponding call
\(((i,r),r)\):
\[
  \mathbf{1}\{V_i\ge r\} \;=\; \mathbf{1}\{V'_{(i,r)}\ge r\}.
\]

Fix any set \(S\) of calls in the original instance and any execution order.
Map \(S\) to \(S'\subseteq\Calls'\) by replacing each \((i,r)\in S\) with its
corresponding call \(((i,r),r)\). Under the coupling above, the accept/reject
outcomes along the execution are identical step-by-step, and therefore the
(realized and expected) revenue is the same for every horizon realization. Hence
the multi-call instance is equivalent to a correlated-values instance with at
most one call per agent.
\end{proof}

\noindent
\refstepcounter{subsection}
\subsection*{Computational Model}
\label{subsec:input-representation-and-computational-model}
We first describe the finite input representation in the independent-values
setting. There are \(M\) agents. Each agent \(i\) has a value distribution
\(F_i\) supported on \(K\) rational values, given explicitly by their
probability masses. The process has at most \(T\) potential call opportunities,
and the horizon random variable \(N\) is specified by its survival
probabilities \(q_t := \Pr[N \ge t]\) for \(t=1,\ldots,T\). Each call
\(e=(i_e,r_e)\) is represented by the index \(i_e\) and a rational price \(r_e\).
The size of the instance is the total encoding length of
\((M, K, T, \{F_i\}_{i\in I}, \{q_t\}_{t=1}^T)\), and all running times are
measured with respect to this encoding.

As noted formally in Section \ref{sec:model}, for correlated valuations the only change to the input is access to samples of the joint distribution of valuations.

\section{Appendix - Section \ref{sec:known}}
\label{sec:app-known}
This appendix contains examples, proof sketches, and formal proofs for the results in Section~\ref{sec:known}

\refstepcounter{subsection}
\subsection*{Examples}
\label{sec:app-known-examples}
\begin{example}[Non-monotonicity and non-submodularity]
\label{ex:nonmono-nonsubmod}
Consider three calls as elements indexed by agent, price and time $(i,r,t)$:
\[
e=(A,5,1),\qquad e'=(B,0.1,2),\qquad f=(C,10,3).
\]
Let the acceptance probabilities be
\[
p_e=0.9,\qquad p_{e'}=0.99,\qquad p_f=0.9,
\]
and assume a deterministic horizon of length at least $3$ (equivalently,
$q_1=q_2=q_3=1$).

\noindent
\textbf{Non-monotonicity:}
Let $S=\{f\}$. Then
\[
\RevUnk{\pi(S)}=9
\qquad\text{while}\qquad
\RevUnk{\pi(S\cup\{e\})}=5.4,
\]
so adding $e$ decreases expected revenue.

\noindent
\textbf{Non-submodularity:}
Let $T=\{e',f\}$ (so $S\subseteq T$). Then
\[
\RevUnk{\pi(T)}=0.189
\qquad\text{and}\qquad
\RevUnk{\pi(T\cup\{e\})}=4.5189,
\]
\[
\RevUnk{\pi(S\cup\{e\})}-\RevUnk{\pi(S)}=-3.6
\;<\;
4.3299=\RevUnk{\pi(T\cup\{e\})}-\RevUnk{\pi(T)}.
\]
\end{example}

\refstepcounter{subsection}
\subsection*{Formal Proofs}
\label{sec:app-known-proofs}

\begin{proof}[Proof of Lemma \ref{lem:price-order-optimal}]
Consider a valuation realization \(\mathbf{v}\in\mathbb{R}_{\ge 0}^{I}\).
If no call in \(S\) is accepted under \(\mathbf{v}\), then
\(\RevV{\infty}{\pi}{\mathbf{v}}=0\) for every permutation \(\pi\) of \(S\), and
the claim holds.

Otherwise, let \(e^\star\in S\) be an accepted call of maximum price under
\(\mathbf{v}\):
\[
  e^\star \in \arg\max\{\, r_e : e\in S,\ \mathbf{v}_{i_e}\ge r_e \,\},
  \qquad r^\star := r_{e^\star}.
\]
Consider the decreasing-price sequence \(\pi^{\mathrm{d}}(S)\). By definition,
every call appearing before \(e^\star\) in \(\pi^{\mathrm{d}}(S)\) has price at
least \(r^\star\). Since \(r^\star\) is the maximum accepted price under
\(\mathbf{v}\), all such earlier calls are rejected under \(\mathbf{v}\). Hence
the first acceptance under \(\pi^{\mathrm{d}}(S)\) occurs at \(e^\star\), and
therefore
\[
  \RevV{\infty}{\pi^{\mathrm{d}}(S)}{\mathbf{v}} \;=\; r^\star.
\]

Now consider an arbitrary permutation \(\pi\) of \(S\). If \(\pi\) has no accepted
call under \(\mathbf{v}\), then \(\RevV{\infty}{\pi}{\mathbf{v}}=0\le r^\star\).
Otherwise, its realized revenue equals the price \(r_e\) of the first accepted
call \(e\in S\) under \(\pi\). By maximality of \(r^\star\) among accepted calls,
we have \(r_e\le r^\star\). In all cases,
\[
  \RevV{\infty}{\pi}{\mathbf{v}}
  \;\le\;
  r^\star
  \;=\;
  \RevV{\infty}{\pi^{\mathrm{d}}(S)}{\mathbf{v}}.
\]
\end{proof}

\begin{lemma}[Decreasing-price order is pointwise monotone and submodular]
\label{lem:dec-reachable-submod}
Consider the decreasing-price order \(\orddown\) on \(\Calls\).
For any valuation realization \(\mathbf{v}\in\mathbb{R}_{\ge 0}^{I}\), the set
function \(\RevV{\infty}{\pi^{\mathrm{d}}(S)}{\mathbf{v}}\) on \ $S$ is monotone
and submodular.
\end{lemma}

\begin{proof}[Proof of Lemma \ref{lem:dec-reachable-submod}]
Consider a valuation realization \(\mathbf{v}\in\mathbb{R}_{\ge 0}^{I}\) and let
\[
  R_{\mathbf{v}}(S)
  \;:=\;
  \RevV{\infty}{\pi^{\mathrm{d}}(S)}{\mathbf{v}}
  \qquad\text{for } S\subseteq\Calls.
\]
Since \(\infty\) indicates that the horizon does not truncate the sequence,
executing \(\pi^{\mathrm{d}}(S)\) selects the first accepted call in decreasing
price order. Equivalently,
\[
  R_{\mathbf{v}}(S)
  \;=\;
  \max_{e\in S}\Bigl\{\, r_e\cdot \mathbf{1}\{\mathbf{v}_{i_e}\ge r_e\}\Bigr\},
\]
with the convention that the maximum over an empty set is \(0\).

\smallskip
\noindent\emph{Monotonicity:}
Consider \(S\subseteq T\). Then
\[
  R_{\mathbf{v}}(T)
  \;=\;
  \max_{e\in T}\Bigl\{\, r_e\cdot \mathbf{1}\{\mathbf{v}_{i_e}\ge r_e\}\Bigr\}
  \;\ge\;
  \max_{e\in S}\Bigl\{\, r_e\cdot \mathbf{1}\{\mathbf{v}_{i_e}\ge r_e\}\Bigr\}
  \;=\;
  R_{\mathbf{v}}(S).
\]

\noindent\emph{Submodularity:}
Consider \(S\subseteq T\subseteq\Calls\) and \(e\notin T\). Using the max
characterization, adding \(e\) affects the value only through whether \(e\) is
accepted and whether its price exceeds the current maximum. Thus,
\[
  R_{\mathbf{v}}(U\cup\{e\})-R_{\mathbf{v}}(U)
  \;=\;
  \begin{cases}
    0,
    & \text{if }\mathbf{v}_{i_e}<r_e \text{ or }
      r_e \le R_{\mathbf{v}}(U),\\[2pt]
    r_e - R_{\mathbf{v}}(U),
    & \text{if }\mathbf{v}_{i_e}\ge r_e \text{ and }
      r_e > R_{\mathbf{v}}(U),
  \end{cases}
  \qquad\text{for any }U\subseteq\Calls.
\]
Since \(S\subseteq T\) implies \(R_{\mathbf{v}}(S)\le R_{\mathbf{v}}(T)\),
comparing the above expression for \(U=S\) and \(U=T\) gives
\[
  R_{\mathbf{v}}(S\cup\{e\})-R_{\mathbf{v}}(S)
  \;\ge\;
  R_{\mathbf{v}}(T\cup\{e\})-R_{\mathbf{v}}(T),
\]
which proves submodularity.
\end{proof}

\begin{proof}[Proof of Theorem \ref{thm:known-horizon-reduction}]
Consider a deterministic horizon \(n\in\mathbb{N}\). By
Lemma~\ref{lem:price-order-optimal}, for any call set \(S\subseteq\Calls\) and
any valuation realization \(\mathbf{v}\in\mathbb{R}_{\ge 0}^{I}\), executing \(S\)
in decreasing-price order \(\pi^{\mathrm{d}}(S)\) is optimal among all orderings
of \(S\) on the sample path:
\[
  \RevV{\infty}{\pi^{\mathrm{d}}(S)}{\mathbf{v}}
  \;=\;
  \max_{\pi\ \text{a permutation of }S}\ \RevV{\infty}{\pi}{\mathbf{v}}.
\]
The use of \(\infty\) is without loss of generality, since for any permutation
\(\pi\) of \(S\), the realized revenue is unaffected once the horizon exceeds
\(|\pi|\). Therefore, the known-horizon problem may restrict attention to
solutions that execute any selected set \(S\) via \(\pi^{\mathrm{d}}(S)\).

Under this fixed order, Lemma~\ref{lem:dec-reachable-submod} implies that for
every valuation realization \(\mathbf{v}\), the quantity
\(\RevV{\infty}{\pi^{\mathrm{d}}(S)}{\mathbf{v}}\) is monotone and submodular in
\(S\). Now incorporate the horizon constraint. For any set \(S\) with \(|S|\le n\),
the decreasing-price sequence \(\pi^{\mathrm{d}}(S)\) has length \(|S|\), and hence
\[
  \RevV{n}{\pi^{\mathrm{d}}(S)}{\mathbf{v}}
  \;=\;
  \RevV{\infty}{\pi^{\mathrm{d}}(S)}{\mathbf{v}}
  \qquad\text{for all }\mathbf{v}\in\mathbb{R}_{\ge 0}^{I}.
\]
Taking expectation over \(V\) yields
\[
  \RevN{n}{\pi^{\mathrm{d}}(S)}
  \;=\;
  \RevN{\infty}{\pi^{\mathrm{d}}(S)}
  \qquad\text{for all }S\subseteq\Calls \text{ with }|S|\le n,
\]
and therefore \(\RevN{n}{\pi^{\mathrm{d}}(S)}\) is monotone and submodular on the
feasible domain \(\{S\subseteq\Calls:\ |S|\le n\}\).

It remains to note that the feasibility constraints in the known-horizon problem
form a matroid. The horizon constraint \(|S|\le n\) is a uniform matroid. The
restriction that each agent is approached at most once is a partition matroid.
Their intersection is a matroid, and thus the feasible family is a matroid.
Consequently, the known-horizon sequential pricing problem is a monotone
submodular maximization problem under a matroid constraint.
\end{proof}

\begin{proof}[Proof of Theorem \ref{thm:hardness-correlated-1-1e}]
We give a polynomial-time, approximation-preserving reduction from weighted
Max-$k$-Coverage.

Let $U$ be a universe with weights $w(u)\ge 0$ and sets
$\{\mathcal{S}_i\}_{i\in[m]}$, and let $k$ be the cardinality budget. The
weighted Max-$k$-Coverage problem is
\[
  \max_{S\subseteq[m],\,|S|\le k}\ \sum_{u\in \cup_{i\in S}\mathcal{S}_i} w(u).
\]
Let $W:=\sum_{u\in U} w(u)$.

From this instance we construct a correlated-values known-horizon pricing
instance with horizon $n=k$ as follows. For each $i\in[m]$, create a distinct
agent $i$ and a single call $e_i=(i,1)$ with posted price $r_{e_i}=1$.

\textbf{Correlated valuation distribution:}
We specify a joint distribution over $V=(V_i)_{i\in[m]}\in\{0,1\}^m$ by giving
its support and probabilities. For each $u\in U$, define a valuation vector
$v^{(u)}\in\{0,1\}^m$ by
\[
  v^{(u)}_i \;:=\; \mathbf{1}\{u\in \mathcal{S}_i\},\qquad i\in[m],
\]
and set
\[
  \Pr\!\bigl[V=v^{(u)}\bigr] \;:=\; \frac{w(u)}{W},\qquad u\in U.
\]
This distribution has support size $|U|$, and thus the pricing instance has a polynomial-size description.

Now any feasible pricing solution selects a subset of at most $k$ calls from
$\{e_1,\ldots,e_m\}$. Ignoring order yields a feasible index set
$S\subseteq[m]$ with $|S|\le k$, and conversely any such $S$ yields a feasible
pricing plan by selecting $\{e_i:i\in S\}$ in any order.

Moreover, since all prices equal $1$, for any $|S|\le k$ and any ordering $\pi$
of $\{e_i:i\in S\}$, every valuation realization $v$ satisfies
\[
  \RevV{k}{\pi}{v}
  \;=\;
  \mathbf{1}\Bigl(\exists\, i\in S \text{ with } v_i\ge 1\Bigr).
\]
In particular, under $v^{(u)}$ this event occurs iff
$u\in \cup_{i\in S}\mathcal{S}_i$, and therefore
\[
  \RevN{k}{\pi}
  \;=\;
  \sum_{u\in U} \Pr[V=v^{(u)}]\cdot
  \mathbf{1}\Bigl(u\in \bigcup_{i\in S}\mathcal{S}_i\Bigr)
  \;=\;
  \frac{1}{W}\sum_{u\in \cup_{i\in S}\mathcal{S}_i} w(u).
\]
Thus, for every feasible $S$, the expected revenue of the corresponding pricing
plan equals the weighted coverage value scaled by the positive constant $1/W$.
Therefore, the reduction preserves approximability: any
$\alpha$-approximation for the pricing instance yields an $\alpha$-approximation
for weighted Max-$k$-Coverage.

The claimed inapproximability bound for the pricing problem follows from
\citet{feige1998threshold}.
\end{proof}

We note that the construction above is closely related in spirit to reductions for
``expected-maximum'' objectives in other application settings; for example, ~\cite{kleinberg2015teamperformancetestscores} use a similar
indicator-based encoding to establish NP-hardness of their optimization
problem.

\refstepcounter{subsection}
\subsection*{Additional Results for the multi-call model}
\label{sec:app-known-multicall}

\begin{theorem}[Multi-call known horizon]
\label{thm:multicall-known}
Consider the deterministic horizon \(n\in\mathbb{N}\). Let \(S_{\mathrm{greedy}}\) be
the output of Algorithm~\ref{alg:greedy-multicall}. Then
\[
  \RevN{n}{\pi^{\mathrm{d}}(S_{\mathrm{greedy}})}
  \;\ge\;
  (1-1/e)\,\RevN{n}{\pi^{\mathrm{d}}(S^\star)},
\]
where \(S^\star\in\arg\max\{\RevN{n}{\pi^{\mathrm{d}}(S)}:\ |S|\le n\}\).
\end{theorem}

\begin{proof}
By Lemma~\ref{lem:multicall-special-case}, any multi-call instance can be viewed
as a correlated-values instance with at most one call per agent. In the resulting representation, each expanded agent has exactly one
associated call (equivalently, only one available price), so the constraint
\emph{``at most one call per agent''} is satisfied automatically by every set of
calls. Hence there is no additional partition-matroid feasibility constraint;
the only feasibility restriction in the known-horizon setting is the horizon
constraint \(|S|\le n\), i.e., a cardinality constraint.

Under the fixed order \(\orddown\), Theorem~\ref{thm:known-horizon-reduction}
implies that \(\RevN{n}{\pi^{\mathrm{d}}(S)}\) is a monotone submodular function
of \(S\) on \(|S|\le n\). Algorithm~\ref{alg:greedy-multicall} is exactly the
standard greedy algorithm applied to this function under a cardinality budget
\(n\): it repeatedly adds an element maximizing the marginal gain
\(\RevN{n}{\pi^{\mathrm{d}}(S\cup\{e\})}-\RevN{n}{\pi^{\mathrm{d}}(S)}\).
Therefore, by the classic result of ~\cite{nemhauser1978analysis},
the greedy algorithm attains a \((1-1/e)\) approximation under a cardinality
constraint, yielding the stated guarantee.
\end{proof}

\begin{algorithm}[t]
\caption{Greedy for Multi-Call Known Horizon}
\label{alg:greedy-multicall}
\DontPrintSemicolon
\SetKwInOut{Input}{Input}
\SetKwInOut{Output}{Output}

\Input{call set \(\Calls\), horizon \(n\)}
\Output{set \(S_{\mathrm{greedy}}\subseteq\Calls\) with \(|S_{\mathrm{greedy}}|\le n\)}

\(S \gets \emptyset\)\;
\For{\(t \gets 1\) \KwTo \(n\)}{
  \(e_t \in \arg\max\limits_{e\in\Calls\setminus S}
     \Bigl(\RevN{n}{\pi^{\mathrm{d}}(S\cup\{e\})}-\RevN{n}{\pi^{\mathrm{d}}(S)}\Bigr)\)\;
  \(S \gets S \cup \{e_t\}\)\;
}
\Return \(S_{\mathrm{greedy}} \gets S\)\;

\end{algorithm}

\section{Appendix - Section \ref{sec:unknown-arbitrary-quarter}}
\label{sec:app-unknown-quarter}
This appendix contains examples, proof intuition, and formal proofs for the results in Section~\ref{sec:unknown-arbitrary-quarter}

\refstepcounter{subsection}
\subsection*{Examples}
\label{sec:app-unknown-quarter-examples}
\begin{example}[Set-dependent optimal order]
\label{ex:set-dependent-order}
Consider three calls to independent agents $A,B,C$, where call $X \in \{A,B,C\}$ is offered at price $r_X$ and is accepted with probability $p_X$

Let the survival probabilities of time be $q_1=1$, $q_2=0.6$, and $q_3=0.2$, and take
\[
(r_A,p_A)=(9,0.5),\qquad (r_B,p_B)=(8,0.6),\qquad (r_C,p_C)=(8,0.8).
\]
On the subset $\{A,B\}$, executing $A$ before $B$ is optimal since
$Rev(A,B)=5.94>Rev(B,A)=5.88$.
However, on the superset $\{A,B,C\}$ the optimal order is $(C,B,A)$ since
$Rev(C,B,A)=7.048>Rev(C,A,B)=7.036$, which implies $B$ should precede $A$.
Thus, adding $C$ flips the optimal relative order of the two existing calls
$A$ and $B$.
\end{example}

\refstepcounter{subsection}
\subsection*{Proof Sketches and Intuition}
\label{sec:app-unknown-quarter-sketch}

\subsubsection{Proof Sketch Lemma \ref{lem:select-mono-det}}
Lemma \ref{lem:select-mono-det} establishes a basic monotonicity property of the pruning rule at a
fixed time.  It says that if an element \(e\) is selected by pruning from a
larger set \(S'\cup\{e\}\), then it must also be selected from any smaller subset
\(S\cup\{e\}\), provided that \(e\) has the largest acceptance probability among
the available offers at that time.  The intuition is that moving from \(S'\) to
\(S\) can only reduce the future value from later times.  When \(e\) has
the highest acceptance probability, a lower future value makes selecting
\(e\) even more attractive relative to all competing options. This property is
a key step toward establishing a reverse-time submodular order for the pruned
evaluator.

\subsubsection{Proof Sketch Lemma \ref{lem:G-subadditive}} Lemma \ref{lem:G-subadditive} shows that the pruned evaluator is subadditive.  The proof proceeds
by viewing \(G(S\cup T)\) as the expected revenue of a single time-feasible plan
selected from the offers in \(S\cup T\).  This plan is then split into two
time-feasible subplans, one using only offers from \(S\) and the other using only
offers from \(T\).  All three plans are coupled on the same realization of the
horizon and valuations.  Under this coupling, any offer that is attempted in the
combined plan must also be attempted in at least one of the two subplans.  As a
result, the realized revenue of the combined plan is bounded by the sum of the
realized revenues of the two subplans.

\refstepcounter{subsection}
\subsection*{Formal Proofs}
\label{sec:app-unknown-quarter-proofs}

\begin{proof}[Proof of Lemma~\ref{lem:envelope-g}]
We prove by contradiction via an exchange argument over time slices.

Let \(S\subseteq\CallsT\).  For each \(t\), let \(\sel_t(S)\) be any maximizer of
the recursion defining \(G_t(S)\); i.e., \(\sel_t(S)=\varnothing\) if \(G_t(S)=0\),
and otherwise \(\sel_t(S)\in S_t\) attains
\(
q_t\pr{e}p_e+(1-p_e)G_{t+1}(S)
\).

By way of contradiction, suppose there exists a time-feasible subset
\(A\subseteq S\) such that
\[
\RevUnk{\pi(A)} \;>\; G_1(S).
\]
Let \(t_m\) be the largest time index at which \(A\) disagrees with the
pruning choice for \(S\).  That is, for all \(t>t_m\), the offer used by
\(\pi(A)\) at time \(t\) equals \(\sel_t(S)\), but at time
\(t_m\) the plan \(\pi(A)\) uses an element \(e_A\neq e^*\), where
\(e^*:=\sel_{t_m}(S)\) is the pruning-chosen element at time \(t_m\).

\noindent
Let \(A_{<t_m}\) denote the prefix of \(A\) before time \(t_m\), and let
\(A_{\ge t_m}\) denote the suffix from time \(t_m\) onward.  Let
\[
p:=\prod_{t'<t_m:\,\sel_{t'}(A)\neq\varnothing}\bigl(1-p_{\sel_{t'}(A)}\bigr)
\]
be the probability that all calls made before \(t_m\) are rejected under
\(\pi(A)\).  (This depends only on the prefix \(A_{<t_m}\).)

Conditioning on the outcomes of times \(<t_m\), the expected revenue decomposes as
\begin{equation}
\label{eq:rev-prefix-suffix}
\RevUnk{\pi(A)}
=
\RevUnk{\pi(A_{<t_m})}
\;+\;
p\cdot \Bigl(q_{t_m}\,\pr{e_A}\,p_{e_A} + (1-p_{e_A})\,\RevUnk{\pi(A_{\ge t_m+1})}\Bigr).
\end{equation}
On the event that all earlier calls reject (probability \(p\)), the
contribution from time \(t_m\) is its immediate expected reward
\(q_{t_m}\pr{e_A}p_{e_A}\), and if \(e_A\) rejects we obtain the continuation
value from times \(>t_m\), namely \(\RevUnk{\pi(A_{\ge t_m+1})}\).

Now we construct a modified subset \(A'\subseteq S\) by replacing the time-\(t_m\)
element \(e_A\) with \(e^*\), and keeping all other time slices identical:
\begin{itemize}[leftmargin=2em]
  \item for \(t<t_m\), let \(A'\) coincide with \(A\);
  \item at time \(t_m\), include \(e^*\) (and exclude \(e_A\));
  \item for \(t>t_m\), let \(A'\) coincide with \(A\).
\end{itemize}
Then \(A'\) is time-feasible, \(A'\subseteq S\), and the prefix (hence \(p\)) is
unchanged.  Moreover \(A'\) agrees with \(A\) on times \(>t_m\), so
\(A'_{\ge t_m+1}=A_{\ge t_m+1}\) and hence
\(\RevUnk{\pi(A'_{\ge t_m+1})}=\RevUnk{\pi(A_{\ge t_m+1})}\).

By definition of the pruning recursion for \(S\), the choice \(e^*=\sel_{t_m}(S)\)
maximizes
\[
q_{t_m}\,\pr{e}\,p_e \;+\; (1-p_e)\,G_{t_m+1}(S)
\]
over \(e\in S_{t_m}\) (and the outside option).  Since \(e_A\in S_{t_m}\), this implies
\begin{equation}
\label{eq:tm-argmax}
q_{t_m}\,\pr{e^*}\,p_{e^*} + (1-p_{e^*})\,G_{t_m+1}(S)
\;\ge\;
q_{t_m}\,\pr{e_A}\,p_{e_A} + (1-p_{e_A})\,G_{t_m+1}(S).
\end{equation}

Finally, since \(A_{\ge t_m+1}\subseteq S_{\ge t_m+1}\) is time-feasible, we have
\begin{equation}
\label{eq:suffix-upperbound}
\RevUnk{\pi(A_{\ge t_m+1})}\ \le\ G_{t_m+1}(S).
\end{equation}
Using \eqref{eq:rev-prefix-suffix} for both \(A\) and \(A'\), together with
\eqref{eq:tm-argmax} and \eqref{eq:suffix-upperbound} (and the equality of suffix
values for \(A\) and \(A'\)), we obtain
\[
\RevUnk{\pi(A')}
\;\ge\;
\RevUnk{\pi(A)}.
\]
Since \(A'\) matches the pruning choice for \(S\) at time \(t_m\) and continues
to match it for all \(t>t_m\), this contradicts the definition of \(t_m\) as the
last time of disagreement for a counterexample maximizing
\(\RevUnk{\pi(\cdot)}\).  Therefore no such \(A\) exists, and hence
\(G_1(S)\ge \max\{\RevUnk{\pi(A)}:\ A\subseteq S,\ A\in \mathcal I_{\mathrm{time}}\}\).
The reverse inequality follows by taking \(A\) induced by the recursion (i.e.,
choosing \(\sel_t(S)\) whenever \(G_t(S)>0\)), and thus \(G_1(S)=G(S)\).
\end{proof}

\begin{lemma}[Selection pattern at time $t$]
\label{lem:select-mono-det}
Let $S\subseteq S'\subseteq\CallsT$, fix a time $t$, and let $e\notin S'$ satisfy $\ti{e}=t$.
Assume that $
p_e \;\ge\; p_x \text{for all } x\in (S'\cup\{e\})_t.
$

If pruning at time $t$ selects $e$ from $S'\cup\{e\}$, then pruning at time $t$ also selects $e$ from $S\cup\{e\}$.
\end{lemma}

\begin{proof}[Proof of Lemma \ref{lem:select-mono-det}]
Introduce a null action $\varnothing$ at time $t$ with $(\pr{\varnothing},p_{\varnothing})=(0,0)$.
For any set $U$ and any $x\in U_t\cup\{\varnothing\}$, define the one-step value
\[
\phi_U(x;t)\ :=\ q_t\,\pr{x}\,p_x \;+\; (1-p_x)\,G_{t+1}(U).
\]
Note that $\phi_U(\varnothing;t)=G_{t+1}(U)$.  Since $S\subseteq S'$, the recursion yields
$G_{t+1}(S')\ge G_{t+1}(S)$.

Fix any competitor $x\in (S'\cup\{e\})_t\cup\{\varnothing\}$. Consider the following difference as a function of the future value $C$:
\[
\big(\phi(e;t)-\phi(x;t)\big)(C)
\;=\; q_t\,( \pr{e}p_e - \pr{x}p_x )\;+\;(p_x - p_e)\,C.
\]
By the assumption $p_e\ge p_x$, the slope with respect to $C$ is $(p_x-p_e)\le 0$, hence this difference is
nonincreasing in $C$.

If $e$ is selected at time $t$ from $S'\cup\{e\}$, then $\phi_{S'}(e;t)\ge \phi_{S'}(x;t)$ for all such competitors $x$.
Evaluating the above difference at $C=G_{t+1}(S')$ and using monotonicity in $C$ together with
$G_{t+1}(S)\le G_{t+1}(S')$ gives $\phi_S(e;t)\ge \phi_S(x;t)$ for all competitors $x$.
Therefore $e$ also attains the argmax at time $t$ under $S\cup\{e\}$.
\end{proof}

\begin{lemma}[Reverse-time submodular order under $p$-tie-breaks]
\label{lem:rev-time-diminish-det}
Let $S\subseteq T\subseteq\CallsT$, and let $e\notin T$ satisfy
\[
\ti{e}\ \le\ \ti{x}\quad \text{for all } x\in S\cup T
\quad(\text{$e$ is weakly the earliest element}).
\]
Assume that within time $\ti{e}$ we break ties by increasing $p$, and that
\[
p_e \;>\; p_x \qquad \text{for all } x\in (T\cup\{e\})_{\ti{e}}.
\]
Then
\[
G(S\cup\{e\}) - G(S) \;\ge\; G(T\cup\{e\}) - G(T).
\]
\end{lemma}

\begin{proof}[Proof of Lemma \ref{lem:rev-time-diminish-det}]
Let $t^\star := \ti{e}$. Because $e$ is weakly the earliest element in $S\cup T$,
the addition of $e$ only affects the recursion at time $t^\star$; all earlier
time slices are unchanged.

For any $U\subseteq\CallsT$, let
\[
C_U \;:=\; G_{t^\star+1}(U)
\]
be the future value from time $t^\star+1$ onward. Introduce the null
action $\varnothing$ at time $t^\star$ with $(\pr{\varnothing},p_{\varnothing})=(0,0)$,
and define the one-step values at time $t^\star$ by
\[
\phi_U(x)
\;:=\;
q_{t^\star}\,\pr{x}\,p_x \;+\; (1-p_x)\,C_U,
\qquad x\in U_{t^\star}\cup\{\varnothing\}.
\]
(Note that $\phi_U(\varnothing)=C_U$.)  Let $\hat x_U$ denote the competitor
selected at time $t^\star$ when $e$ is not present:
\[
\hat x_U \;\in\;
\arg\max\bigl\{
   \phi_U(x):x\in U_{t^\star}\cup\{\varnothing\}
\bigr\}.
\]

By the pruning recursion at time $t^\star$ and that $e$ does not
appear before $t^\star$, the marginal contribution of $e$ to $G$ under $U$ is
\[
\Delta(e\mid U)
\;:=\;
G(U\cup\{e\}) - G(U)
\;=\;
\max\{\,0,\ \phi_U(e) - \phi_U(\hat x_U)\,\}.
\]

Let
\[
\SelG(e;U)
\;:=\;
\mathbf{1}\!\left\{
   \text{$e$ is selected at time $t^\star$ by pruning on }U\cup\{e\}
\right\}.
\]

If $\SelG(e;T)=0$, then $\Delta(e\mid T)=0$, while $\Delta(e\mid S)\ge 0$, hence
$\Delta(e\mid S)\ge \Delta(e\mid T)$ holds.  Moreover, the pattern
$\SelG(e;S)=0$ and $\SelG(e;T)=1$ is impossible by Lemma~\ref{lem:select-mono-det}.
Therefore, it remains to consider the case $\SelG(e;S)=\SelG(e;T)=1$.

In this case, the
marginal contribution of \(e\) is determined by how much better it is than the
best alternative available at time \(t^\star\).  The comparison therefore hinges
on how the pruning scores of \(e\) and its strongest competitor change as we
move from the smaller set \(S\) to the larger set \(T\).

When we increase the future value from \(C_S\) to \(C_T\), the score
\[
\phi_U(x)=q_{t^\star}\pr{x}p_x+(1-p_x)C_U
\]
of each action \(x\) increases by \((1-p_x)(C_T-C_S)\).  Actions with smaller
acceptance probability benefit more from a higher future value, since
they are more likely to continue to the future.  Thus, moving from \(S\) to
\(T\) strengthens the best competing option relative to \(e\). This is the key fact that enables the improvement from selecting \(e\) over its best alternative to be
smaller under the larger set.
Formally,
\[
\Delta(e\mid S) = \phi_S(e) - \phi_S(\hat x_S),
\qquad
\Delta(e\mid T) = \phi_T(e) - \phi_T(\hat x_T),
\]
and thus
\[
\Delta(e\mid S) - \Delta(e\mid T)
=
\bigl(\phi_S(e) - \phi_T(e)\bigr)
+
\bigl(\phi_T(\hat x_T) - \phi_S(\hat x_S)\bigr).
\]
By definition of $\phi_U(e)$,
\[
\phi_S(e) - \phi_T(e)
=
(1-p_e)(C_S - C_T).
\]

Next, we lower bound $\phi_T(\hat x_T) - \phi_S(\hat x_S)$.
Since $\hat x_T$ maximizes $\phi_T(\cdot)$ over $T_{t^\star}\cup\{\varnothing\}$,
\[
\phi_T(\hat x_T) \;\ge\; \phi_T(\hat x_S).
\]
Hence
\[
\phi_T(\hat x_T) - \phi_S(\hat x_S)
\;\ge\;
\phi_T(\hat x_S) - \phi_S(\hat x_S)
=
(1-p_{\hat x_S})(C_T - C_S).
\]

Combining the two bounds gives
\[
\Delta(e\mid S) - \Delta(e\mid T)
\;\ge\;
(1-p_e)(C_S - C_T)
+
(1-p_{\hat x_S})(C_T - C_S)
=
(p_e - p_{\hat x_S})(C_T - C_S).
\]

By monotonicity of $G$ (or Lemma~\ref{lem:envelope-g}), $S\subseteq T$ implies
$C_T = G_{t^\star+1}(T) \ge G_{t^\star+1}(S) = C_S$. Moreover, the assumption
$p_e>p_x$ for all $x\in (T\cup\{e\})_{t^\star}$ implies in particular
$p_e > p_{\hat x_S}$, since $\hat x_S\in S_{t^\star}\cup\{\varnothing\}$
and either $\hat x_S\in S_{t^\star}\subseteq T_{t^\star}$ or $\hat x_S=\varnothing$
(with $p_{\varnothing}=0$). Thus
\[
p_e - p_{\hat x_S} \;>\; 0,
\qquad
C_T - C_S \;\ge\; 0,
\]
and therefore $\Delta(e\mid S) - \Delta(e\mid T)\ge 0$.  Hence
\[
G(S\cup\{e\}) - G(S)
\;=\;
\Delta(e\mid S)
\;\ge\;
\Delta(e\mid T)
\;=\;
G(T\cup\{e\}) - G(T),
\]
\end{proof}

\begin{lemma}[Subadditivity of the pruned evaluator]
\label{lem:G-subadditive}
For all $S,T \subseteq \CallsT$,
\[
  G(S \cup T)
  \;\le\;
  G(S) + G(T).
\]
\end{lemma}
\begin{proof}[Proof of Lemma \ref{lem:G-subadditive}]
Fix arbitrary $S,T \subseteq \CallsT$.
By Lemma~\ref{lem:envelope-g}, there exists a time-feasible
$A^\star \subseteq S \cup T$ such that
\[
  G(S \cup T) \;=\; \RevUnk{\pi(A^\star)}.
\]

Partition $A^\star$ into
\[
  A_S \;:=\; A^\star \cap S,
  \qquad
  A_T \;:=\; A^\star \cap T.
\]
Since $A^\star$ is time-feasible, so are $A_S$ and $A_T$.
By monotonicity of $G$ and Lemma~\ref{lem:envelope-g},
\[
  \RevUnk{\pi(A_S)} \;\le\; G(S),
  \qquad
  \RevUnk{\pi(A_T)} \;\le\; G(T),
\]
so it suffices to show
\(
  \RevUnk{\pi(A^\star)} \le \RevUnk{\pi(A_S)} + \RevUnk{\pi(A_T)}.
\)

Couple all plans on the same realization of the horizon \(N\) and valuations \(V\).
For any time-feasible \(U\subseteq\CallsT\), let \(\Plan_U:=\pi(U)\).
Then, by definition,
\[
  \RevUnk{\pi(U)} \;=\; \mathbb{E}_{V,N}\!\bigl[\RevV{N}{\Plan_U}{V}\bigr].
\]

Fix a realization \((n,v)\) of \((N,V)\).  Since \(A^\star\) is time-feasible,
each of the plans \(\Plan_{A^\star},\Plan_{A_S},\Plan_{A_T}\) makes at most one
offer per time, and stops after the first acceptance.  Hence there exist
indicators \(Y_e(U;n,v)\in\{0,1\}\) such that
\[
  \RevV{n}{\Plan_U}{v}
  \;=\;
  \sum_{e\in U}
     r_e \cdot \mathbf{1}\{n\ge \ti{e}\}\cdot \mathbf{1}\{v_{\ag{e}}\ge r_e\}
     \cdot Y_e(U;n,v),
\]
where \(Y_e(U;n,v)=1\) iff call \(e\) is actually attempted by \(\Plan_U\) under \((n,v)\).

Construct a partition of \(A^\star\) into \(A_S\) and \(A_T\) (if an element lies
in \(S\cap T\), assign it arbitrarily to one of them).
Let \(e\in A^\star\), and suppose \(e\) is assigned to \(A_S\).
If \(Y_e(A^\star;n,v)=1\), then under \((n,v)\) all calls scheduled before \(e\)
in \(\Plan_{A^\star}\) must have been rejected and \(n\ge \ti{e}\).
Since \(A_S\subseteq A^\star\), the set of calls from \(A_S\) scheduled before \(e\)
is a subset of those scheduled before \(e\) in \(A^\star\); thus the same
rejections imply \(Y_e(A_S;n,v)=1\).
Therefore,
\[
  Y_e(A^\star;n,v)\ \le\ Y_e(A_S;n,v)
  \qquad\text{for all }e\in A^\star\text{ assigned to }A_S.
\]
An identical argument holds for elements assigned to \(A_T\).
Summing over \(e\in A^\star\) and using nonnegativity of each summand yields
\[
  \RevV{n}{\Plan_{A^\star}}{v}
  \;\le\;
  \RevV{n}{\Plan_{A_S}}{v} \;+\; \RevV{n}{\Plan_{A_T}}{v}.
\]

Taking expectations over \((V,N)\) gives
\[
  \RevUnk{\pi(A^\star)}
  \;=\;
  \mathbb{E}_{V,N}\!\bigl[\RevV{N}{\Plan_{A^\star}}{V}\bigr]
  \;\le\;
  \mathbb{E}_{V,N}\!\bigl[\RevV{N}{\Plan_{A_S}}{V}\bigr]
  \;+\;
  \mathbb{E}_{V,N}\!\bigl[\RevV{N}{\Plan_{A_T}}{V}\bigr]
\]
\[
  \;=\;
  \RevUnk{\pi(A_S)} + \RevUnk{\pi(A_T)}
  \;\le\;
  G(S)+G(T).
\]
Using \(G(S\cup T)=\RevUnk{\pi(A^\star)}\) completes the proof.
\end{proof}

\begin{proof}[Proof of Theorem \ref{thm:single-unknown-quarter}]
Let $G$ be the pruning evaluator from
Section~\ref{subsec:prune-def}, and let $\sigma^{\mathrm{rev}}$ be the order that
sorts elements of $\CallsT$ by decreasing time index, breaking ties within each
time slice by increasing acceptance probability $p_e$.
Let
$\mathcal I_{\mathrm{time}}:=\{A\subseteq\CallsT:\forall t\in[T],\,|\{e\in A:\ti(e)=t\}|\le 1\}$
and
$\mathcal I_{\mathrm{agent}}:=\{A\subseteq\CallsT:\forall i\in I,\,|\{e\in A:\ag(e)=i\}|\le 1\}$.

By Lemma~\ref{lem:envelope-g}, for every set $S\subseteq\CallsT$,
\[
G(S)=\max_{A\subseteq S,\ A\in\mathcal I_{\mathrm{time}}}\RevUnk{\pi(A)},
\qquad\text{and in particular}\qquad
G(S)=\RevUnk{\pi(\prune(S))}.
\]
Therefore,
\[
\RevUnk{\pi(\prune(S_{\mathrm{alg}}))}
\;=\;
G(S_{\mathrm{alg}}).
\]

By Lemma~\ref{lem:envelope-g}, $G$ is monotone; by Lemma~\ref{lem:G-subadditive},
$G$ is subadditive; and by Lemma~\ref{lem:rev-time-diminish-det}, $G$ is
$\sigma^{\mathrm{rev}}$-submodular ordered in the sense of
Definition~\ref{def:submodular-order}. The single-call feasibility constraint
(at most one call per agent) is the partition matroid $\mathcal I_{\mathrm{agent}}$.
Hence the directed local-search algorithm of \citet[Algorithm~9]{pmlr-v202-udwani23a},
instantiated with objective $G$ and order $\sigma^{\mathrm{rev}}$, returns
$S_{\mathrm{alg}}$ satisfying
\[
G(S_{\mathrm{alg}})\;\ge\;\tfrac14 \max_{S\in\mathcal I_{\mathrm{agent}}} G(S).
\]

Finally, by Lemma~\ref{lem:envelope-g} (take $S=\CallsT$ and restrict to
$\mathcal I_{\mathrm{agent}}$ subsets), the benchmark satisfies
\[
\max_{S\in\mathcal I_{\mathrm{agent}}} G(S)
\;=\;
\max_{\Plan\in\mathcal I_{\mathrm{agent}}}\RevUnk{\Plan}
\;=\;OPT.
\]
Combining them gives
\(
\RevUnk{\pi(S_{\mathrm{alg}})}\ge \tfrac14\,OPT
\),
as claimed.
\end{proof}

\refstepcounter{subsection}
\subsection*{Algorithms}
\label{sec:app-unknown-quarter-alg}

\begin{algorithm}[t]
\caption{Directed Local Search for Independent Valuations and Unknown Horizon}
\label{alg:single-unknown-local-one-set}
\DontPrintSemicolon
\KwIn{Time-indexed call set $\CallsT$; reverse-time order $\sigma^{\mathrm{rev}}$}
\KwOut{A set $S\subseteq \CallsT$ with at most one call per agent and at most one call per time}

\smallskip
\noindent\emph{Notation.} Throughout, we evaluate sets by executing them in time order.
Accordingly, we write $\RevUnk{S}$ as shorthand for the expected revenue of the
time-ordered plan $\pi(S)$.

\BlankLine
Sort $\CallsT=\{e_1,\ldots,e_L\}$ in increasing $\sigma^{\mathrm{rev}}$ order\;
$S\gets \emptyset$\;
$v_i\gets 0$ for all agents $i$\tcp*{stored value for the currently held call of agent $i$}
\BlankLine

\For{$j=1$ \KwTo $L$}{
  $e\gets e_j$;\quad $t\gets \ti(e)$;\quad $i\gets \ag(e)$\;

  let $h$ be the (unique) call in $S$ with time of $h=t$, if it exists; otherwise set $h=\varnothing$\;
  let $f$ be the (unique) call in $S$ with agent of $f=i$, if it exists; otherwise set $f=\varnothing$\;

  $R_{\text{old}}\gets \RevUnk{S}$\;
  $S^{\text{swap}}\gets (S\setminus\{h\})\cup\{e\}$\tcp*{make $e$ the active call at time $t$}
  $\Delta \gets \RevUnk{S^{\text{swap}}}-R_{\text{old}}$\;

  \If{$\Delta \le 0$}{continue\;}

  \uIf{$f=\varnothing$}{
      $S \gets S^{\text{swap}}$\;
      $v_i \gets \Delta$\;
  }
  \Else{
      \If{$\Delta > v_i$}{
          $S \gets S^{\text{swap}}\setminus\{f\}$\tcp*{enforce the agent constraint}
          $v_i \gets v_i+\Delta$\;
      }
  }
}
\Return{$S$}\;
\end{algorithm}

\section{Appendix - Section \ref{sec:structural-relaxation}}
\label{sec:app-unknown-half}

This appendix contains examples, formal proofs and algorithms for the results in Section~\ref{sec:structural-relaxation}

\refstepcounter{subsection}
\subsection*{Formal Proofs}
\label{sec:app-unknown-half-proofs}

\begin{lemma}[Relaxation is an upper bound]
\label{lem:time-milp-upper}
The optimal value of~\eqref{eq:time-milp}
is an upper bound on the expected revenue
achievable by any feasible solution.
\end{lemma}

\begin{proof}[Proof of Lemma \ref{lem:time-milp-upper}]
Consider any feasible solution that specifies a feasible set of calls
\(S \subseteq \CallsT\) together with an execution order \(\pi\) on \(S\).
For each \(e=(i,r,t)\), define
\[
x_e := \mathbf{1}\{e \in S\},
\qquad
y_e := \Pr[\text{$e$ is the successful call under }(S,\pi)].
\]
Then
\[
0 \le y_e \le \Pr[\text{$e$ succeeds}]\,x_e
   = p_e\,x_e,
\qquad
\sum_{e\in\CallsT} y_e \le 1.
\]
The expected revenue of the policy satisfies
\[
\mathbb{E}[\Rev{\pi}{S}{V}{N}]
   = \sum_{e\in\CallsT} q_{\ti{e}}\,v_e\,y_e.
\]
Hence, the vector pair \((x,y)\)
constructed from any feasible policy
satisfies all constraints of~\eqref{eq:time-milp}
and attains objective value
\(\sum_e q_{\ti{e}} v_e y_e\).
Therefore,
\[
\max_{x,y}\ \sum_e q_{\ti{e}} v_e y_e
\ \ge\
\mathbb{E}[\Rev{\pi}{S}{V}{N}]
\quad
\text{for all feasible policies }(S,\pi),
\]
which proves that~\eqref{eq:time-milp}
is an upper bound on the achievable expected revenue.
\end{proof}

\begin{proof}[Proof of Lemma \ref{lem:breakpoint}]
Consider the augmented instance containing the terminal dummy call \(e_\dagger\).
Fix any feasible \(x\in\mathcal I_1\cap\mathcal I_2\). Conditional on \(x\), the
optimization over \(y\) becomes
\[
\max\Bigl\{\sum_{e} c_e y_e:\ \sum_{e} y_e\le 1,\ 0\le y_e\le p_e x_e\ \forall e\Bigr\},
\]
which is a fractional knapsack problem with item values \(c_e\) and capacities
\(p_e x_e\). Therefore there exists an optimal \(y\) that fills mass greedily in
nonincreasing \(c_e\): it saturates every item strictly above a breakpoint
density, possibly takes one item partially, and assigns zero mass below. This
gives the claimed density-threshold form.

For (i), take any optimal \((x,y)\). If \(\sum_e y_e=1\) we are done. Otherwise
\(\sum_e y_e<1\). Since \(e_\dagger\) is always feasible when included
(\(x_{e_\dagger}=1\)) and has capacity \(p_{e_\dagger}=1\), we can increase
\(y_{e_\dagger}\) by \(1-\sum_e y_e\) to make the constraint tight without
changing the objective (because \(q_{t_\dagger}=0\)). Hence an optimal solution
exists with \(\sum_e y^\star_e=1\).
\end{proof}

\begin{proof}[Proof of Lemma \ref{lem:time-alg-ptas}]
We work in the augmented instance with the terminal dummy call so that, by
Lemma~\ref{lem:breakpoint}, there exists an optimal solution
\((x^\star,y^\star)\) to~\eqref{eq:time-milp} with \(\sum_e y^\star_e=1\) and a
breakpoint element \(g\); set \(\tau:=c_g\).

\noindent\emph{Step 1: Exact reduction of the MILP to a contracted above-threshold problem.}
Consider the contracted ground set obtained by \emph{fixing} \(x_g=1\) and
deleting all calls that share agent or time with \(g\). On this contracted
instance, define the budgeted matroid intersection objective
\[
OPT(\tau)\;:=\;
\max\Bigl\{\sum_{e\in H} p_e(c_e-\tau):
\ H\in\mathcal I_1\cap\mathcal I_2,\ H\subseteq\{e:\ c_e\ge\tau\},\
\sum_{e\in H} p_e\le 1\Bigr\}.
\]
Let
\[
H^\star \;:=\; \{\,e\in\CallsT:\ x^\star_e=1,\ c_e>\tau\,\}.
\]
By Lemma~\ref{lem:breakpoint}, we may assume \(y^\star_e=p_ex^\star_e\) for
all \(e\) with \(c_e>\tau\) and \(y^\star_e=0\) for all \(e\) with \(c_e<\tau\).
Hence \(H^\star\) is feasible in the contracted instance, satisfies
\(\sum_{e\in H^\star} p_e\le 1\), and therefore is feasible for \(OPT(\tau)\),
implying
\begin{equation}
\label{eq:opt-tau-lb-simple}
OPT(\tau)\ \ge\ \sum_{e\in H^\star} p_e(c_e-\tau).
\end{equation}

Moreover, the breakpoint form also shows that the MILP optimum decomposes as
\begin{equation}
\label{eq:opt-decomp-simple}
OPT_{\mathrm{MILP}}
=\sum_e c_e y^\star_e
=\tau+\sum_{e\in H^\star} p_e(c_e-\tau).
\end{equation}
In particular, \eqref{eq:opt-tau-lb-simple}--\eqref{eq:opt-decomp-simple}
identify \(\tau\) as the contribution of the breakpoint mass placed on density
\(\tau\), and the remaining value as exactly the contracted above-threshold
surplus.

\noindent\emph{Step 2: The PTAS output achieves \((1-\varepsilon)OPT_{\mathrm{MILP}}\).}
Let \(H^{ALG}\) be the set returned in Step~3 of
Algorithm~\ref{alg:time-lowdensity}. By the PTAS guarantee,
\[
\sum_{e\in H^{ALG}} p_e(c_e-\tau)
\ \ge\
(1-\varepsilon)\,OPT(\tau).
\]
Algorithm~\ref{alg:time-lowdensity} sets \(y_e=p_e\) for all \(e\in H^{ALG}\)
and assigns the remaining mass to \(g\) (and, if needed, to the terminal dummy
call, which contributes zero to the objective). Hence
\begin{equation}
\label{eq:alg-decomp-simple}
\sum_e c_e y^{ALG}_e
\ \ge\
\tau+\sum_{e\in H^{ALG}} p_e(c_e-\tau).
\end{equation}
Combining \eqref{eq:alg-decomp-simple} with the PTAS guarantee and then with
\eqref{eq:opt-tau-lb-simple} yields
\[
\sum_e c_e y^{ALG}_e
\ \ge\
\tau + (1-\varepsilon)\sum_{e\in H^\star} p_e(c_e-\tau).
\]
Since \(\tau\ge 0\), using \eqref{eq:opt-decomp-simple} we conclude
\[
\sum_e c_e y^{ALG}_e
\ \ge\
(1-\varepsilon)\Bigl(\tau+\sum_{e\in H^\star} p_e(c_e-\tau)\Bigr)
=(1-\varepsilon)OPT_{\mathrm{MILP}},
\]
as claimed.
\end{proof}

\begin{proof}[Proof of Lemma \ref{lem:half-attenuated-factored}]
Write \(\pi(x)=(e_1,\ldots,e_k)\).  For each \(t\), let
\[
  Y_{t-1} := \sum_{s<t} y_{e_s},
  \qquad
  \gamma_t := \Pr[\text{no call accepts among }e_1,\ldots,e_{t-1} ].
\]
Since \((x,y)\) is feasible, \(0\le y_{e_t}\le p_{e_t}\), hence
\(\alpha_{e_t}=y_{e_t}/p_{e_t}\le 1\).  Also \(Y_{t-1}\le \sum_e y_e\le 1\), so
\(1-\tfrac12 Y_{t-1}\ge \tfrac12\), implying \(\beta_{e_t}\le 1\); clearly
\(\beta_{e_t}\ge 0\).

We prove by induction that
\begin{equation}
\label{eq:gamma-invariant}
  \gamma_t \;=\; 1-\tfrac12\,Y_{t-1}.
\end{equation}
For \(t=1\), \(Y_0=0\) and \(\gamma_1=1\).  Assume~\eqref{eq:gamma-invariant}
holds for some \(t\).  Conditional on no earlier acceptance,
Algorithm~\ref{alg:half-attenuated} attempts \(e_t\) with probability
\(\beta_{e_t}\), and then \(e_t\) accepts with probability \(p_{e_t}\).  Hence
\[
  \Pr[\text{successful call is }e_t]
  \;=\;
  \gamma_t \cdot \beta_{e_t}\cdot p_{e_t}
  \;=\;
  \gamma_t \cdot \frac{\tfrac12\,\alpha_{e_t}}{1-\tfrac12\,Y_{t-1}} \cdot p_{e_t}
  \;=\;
  \tfrac12\,\alpha_{e_t}p_{e_t}
  \;=\;
  \tfrac12\,y_{e_t},
\]
where we used~\eqref{eq:gamma-invariant} in the third equality and the
definition \(\alpha_{e_t}=y_{e_t}/p_{e_t}\) in the last.  Therefore
\[
  \gamma_{t+1}
  \;=\;
  \gamma_t - \Pr[\text{successful call is }e_t]
  \;=\;
  \Bigl(1-\tfrac12 Y_{t-1}\Bigr) - \tfrac12 y_{e_t}
  \;=\;
  1-\tfrac12 Y_t,
\]
completing the induction.  This also yields
\(\Pr[\text{$e_t$ is the successful call}]=\tfrac12 y_{e_t}\) for every \(t\).
\end{proof}

\begin{proof}[Proof of Theorem \ref{thm:half-eps-approx}]
We work in the augmented instance including the terminal dummy call, and let the
optimal solution to~\eqref{eq:time-milp} be \((x^\star,y^\star)\) .  Let \(g\) be a
breakpoint element of \((x^\star,y^\star)\).  Consider the execution of
Algorithm~\ref{alg:relax-attenuate} under the guess \(g\), and let
\((x^{ALG},y^{ALG})\) be the feasible pair produced by
Algorithm~\ref{alg:time-lowdensity} on this guess.

By Lemma~\ref{lem:time-alg-ptas},
\[
\sum_{e\in\CallsT} q_{\ti{e}}\,r_e\,y^{ALG}_e
=\sum_{e\in\CallsT} c_e\,y^{ALG}_e
\;\ge\;
(1-\varepsilon)\,\OPT_{\mathrm{MILP}}.
\]
Algorithm~\ref{alg:relax-attenuate} chooses \(g^\star\) maximizing the MILP
objective value among all guesses, and then runs
\textsc{Attenuated-Execute} on \((x^{(g^\star)},y^{(g^\star)})\).  Therefore, it
suffices to analyze the policy obtained by running
\textsc{Attenuated-Execute} on \((x^{ALG},y^{ALG})\).  Let \(\Plan^{\mathrm{ALG}}\)
denote the resulting randomized plan.

Let \(S:=\{e\in\CallsT:x^{ALG}_e=1\}\) be the selected call set, ordered by
increasing designated time, and for each \(e\in S\) let \(F_e\) denote the event
that \(e\) is the successful call under \(\Plan^{\mathrm{ALG}}\).  The events
\((F_e)_{e\in S}\) are disjoint, and on \(F_e\) the call \(e\) is attempted no
later than \(\ti{e}\). Hence, if \(\tau(e)\) is the (random) time at which
\(e\) is attempted on outcome \(F_e\), we have \(\tau(e)\le \ti{e}\) and thus
\[
\Pr[N\ge \tau(e),\,F_e]
\;\ge\;
\Pr[N\ge \ti{e},\,F_e]
=
q_{\ti{e}}\Pr[F_e].
\]
It follows that
\begin{align*}
\mathbb{E}\!\left[\RevUnk{\Plan^{\mathrm{ALG}}}\right]
&=\sum_{e\in S} r_e\cdot \Pr[N\ge \tau(e),\,F_e]\\
&\ge \sum_{e\in S} q_{\ti{e}}\,r_e\,\Pr[F_e].
\end{align*}
By Lemma~\ref{lem:half-attenuated-factored}, \(\Pr[F_e]=\tfrac12\,y^{ALG}_e\) for
each selected call \(e\in S\). Substituting gives
\[
\mathbb{E}\!\left[\RevUnk{\Plan^{\mathrm{ALG}}}\right]
\;\ge\;
\frac12 \sum_{e\in S} q_{\ti{e}}\,r_e\,y^{ALG}_e
=
\frac12 \sum_{e\in\CallsT} q_{\ti{e}}\,r_e\,y^{ALG}_e
\;\ge\;
\tfrac12(1-\varepsilon)\,\OPT_{\mathrm{MILP}}.
\]
Finally, Lemma~\ref{lem:time-milp-upper} implies \(\OPT_{\mathrm{MILP}}\ge \OPT\),
completing the proof.
\end{proof}

\begin{example}[Stochasticity gap of $2$ for the MILP]
\label{ex:milp-stoch-gap-two}
Let $\varepsilon\in(0,1)$, $\delta\in(0,1)$, and an integer $n\ge 1$.
Assume there are $n{+}1$ available calls (one for each independent agent).
Suppose horizon $N$ satisfies $q_1=1$ and $q_2=\cdots=q_{n+1}=\delta$, i.e.,
$N=1$ with probability $1-\delta$ and $N\ge n{+}1$ with probability $\delta$.
Further suppose there is one call $a$ with revenue $r_a=1$ and success probability $p_a=1-\varepsilon$.
And there are $n$ additional calls $b_1,\dots,b_n$, each with revenue $r_{b_j}=m$ and
success probability $p_{b_j}=\varepsilon/n$.

\noindent\emph{MILP value:}
Consider the feasible solution to the mixed-integer relaxation that selects all
calls and assigns
\[
y_a = 1-\varepsilon,
\qquad
y_{b_j} = \varepsilon/n
\quad\text{for all } j=1,\dots,n.
\]
Since $\sum_e y_e = 1$ and $y_e \le p_e$ for all calls, this solution is feasible.
Its objective value is
\[
\OPT_{\mathrm{MILP}}
\;\ge\;
q_1 \cdot r_a y_a
\;+\;
\sum_{j=1}^n q_t \cdot r_{b_j} y_{b_j}
=
(1-\varepsilon) + \delta m \varepsilon.
\]

\noindent\emph{Optimal sequential revenue:}
For parameter choices, $m=1/(\varepsilon\delta)$ with
$\delta=1/\sqrt{n}$ and $\varepsilon$ small, it is optimal to attempt call $a$ in the
first period and, if it is rejected and the horizon continues, to attempt the
remaining calls $b_1,\dots,b_n$ in any order.
The expected revenue of this policy is
\[
\OPT
=
(1-\varepsilon)\cdot 1
\;+\;
\varepsilon \cdot \delta \cdot m
\cdot
\bigl(1-(1-\varepsilon/n)^n\bigr).
\]

Substituting $m=1/(\varepsilon\delta)$ yields
\[
\OPT_{\mathrm{MILP}}
\;\ge\;
(1-\varepsilon) + 1
=
2-\varepsilon,
\]
while
\[
\OPT
=
(1-\varepsilon)
\;+\;
\bigl(1-(1-\varepsilon/n)^n\bigr).
\]
For fixed $\varepsilon$ and $n\to\infty$,
\[
1-(1-\varepsilon/n)^n \;\longrightarrow\; 1-e^{-\varepsilon}.
\]
Hence,
\[
\frac{\OPT_{\mathrm{MILP}}}{\OPT}
\;\ge\;
\frac{2-\varepsilon}
{(1-\varepsilon) + (1-e^{-\varepsilon})}.
\]
As $\varepsilon\to 0$,
\[
\frac{\OPT_{\mathrm{MILP}}}{\OPT}
\;=\; 2.
\]
This shows that the mixed-integer relaxation exhibits a stochasticity gap of $2$.

\end{example}

\refstepcounter{subsection}
\subsection*{Algorithms}
\label{sec:app-unknown-half-alg}

\begin{algorithm}[t]
\caption{\textsc{Attenuated-Execute}\((x,y)\)}
\label{alg:half-attenuated}
\DontPrintSemicolon
\KwIn{A feasible pair \((x,y)\) for~\eqref{eq:time-milp}}
\KwOut{A randomized feasible plan \(\Plan^{\mathrm{ATT}}\)}

Let \(\pi(x)=(e_1,e_2,\ldots,e_k)\) be the time-ordered sequence induced by \(x\).\;
Initialize \(Y \gets 0\).
\For{\(t\gets 1\) \KwTo \(k\)}{
  Set \(\alpha \gets y_{e_t}/p_{e_t}\) and
  \(\beta \gets \dfrac{\tfrac12\,\alpha}{1-\tfrac12\,Y}\).\;
  If no earlier call has accepted, attempt \(e_t\) with probability \(\beta\).\;
  Update \(Y \gets Y + p_{e_t}\alpha\).
}
\end{algorithm}

\section{Appendix - Section \ref{sec:unknown-geom}}
\label{sec:app-geom}

This appendix contains proof intuition and formal proofs for the results in Section~\ref{sec:unknown-geom}

\refstepcounter{subsection}
\subsection*{Definitions}
\label{sec:app-geom-def}

\begin{definition}[Score order]
\label{def:score-order}
For each call $e=(i_e,r_e)\in\Calls$ with success probability
$p_e=\Pr[V_{i_e}\ge r_e]$, define the geometric–horizon score
\[
   \mathrm{Score}_\rho(e)
   := \frac{r_e\,p_e}{\,p_e + (1-p_e)\rho\,}.
\]
The \emph{score order} $\ordscore$ is the total order on $\Calls$ that sorts
calls in nonincreasing $\mathrm{Score}_\rho(e)$, breaking ties
deterministically.
\end{definition}
For any feasible call set $S\subseteq\Calls$, execution proceeds according to
the restriction of $\ordscore$ to $S$.

\refstepcounter{subsection}
\subsection*{Proof Sketches}
\label{sec:app-geom-sketch}

\subsubsection{Proof sketch for Lemma \ref{lem:geom-submod}}
In Lemma \ref{lem:geom-submod}, we establish submodularity of expected revenue under the score order. To show this, we compare
the marginal contribution of a call \(e\) when added to two sets
\(A\subseteq B\). We let \(B=A\cup\{e'\}\) for some call \(e'\), the proof reduces
to analyzing how the presence of \(e'\) affects the probability and payoff of
reaching \(e\).

The argument proceeds by conditioning on the relative position of \(e'\) and
\(e\) in the order \(\ordscore\). If \(e'\) precedes \(e\), then either \(e'\)
accepts—blocking \(e\) entirely—or it fails, in which case the geometric horizon
imposes a multiplicative factor \(\rho\) on the probability of reaching \(e\),
reducing its expected marginal contribution. If \(e'\) succeeds \(e\), then the
probability of reaching \(e\) is identical under \(A\) and \(B\), and we compare
marginals by conditioning on whether \(e\) itself accepts or fails; in both
cases, the continuation value under \(B\) dominates that under \(A\), again
implying a smaller marginal. In all cases, the expected marginal contribution of
\(e\) is no larger when added to \(B\) than when added to \(A\), establishing
submodularity.

\refstepcounter{subsection}
\subsection*{Formal Proofs}
\label{sec:app-geom-proofs}

\begin{lemma}[Score order exists and is optimal]
\label{lem:geom-score-global}
Under independence of success events and the geometric-horizon model described above,
there exists a total order ~$\ordscore$ on~$\mathcal{C}$
such that for every feasible subset~$S$, the optimal execution order of calls in~$S$
is obtained by restricting~$\ordscore$ to~$S$.
\end{lemma}

Lemma \ref{lem:geom-score-global} follows directly from Theorem~5 of~\citet{brubach2023onlinematchingframeworksstochastic}.
We have restated it here using our notation. We show via Lemma \ref{lem:geom-monotone} and \ref{lem:geom-submod} that under this optimal score order, the expected revenue function is monotone
and submodular.

\begin{lemma}[Monotonicity of the score-order revenue]
\label{lem:geom-monotone}
Let $\rho\in(0,1)$ and let $\ordscore$ be the global score order.
Then for all $S\subseteq\mathcal{C}$ and all $e\notin S$,
$
   \RevUnk{\Plan(\ordscore,S\cup\{e\})}
   \;\ge\;
   \RevUnk{\Plan(\ordscore,S)}.
$
\end{lemma}

\begin{lemma}[Submodularity of the score-order revenue]
\label{lem:geom-submod}
Let $\rho\in(0,1)$ and the global score order $\ordscore$.
Throughout the proof, all plans execute offers according to $\ordscore$; thus we
write $\Plan(S)$ as shorthand for $\Plan(\ordscore,S)$.

For all $A\subseteq B\subseteq\mathcal{C}$ and all $e\notin B$,
\[
   \RevUnk{\Plan(A\cup\{e\})} - \RevUnk{\Plan(A)}
   \;\ge\;
   \RevUnk{\Plan(B\cup\{e\})} - \RevUnk{\Plan(B)}.
\]
\end{lemma}

\begin{proof}[Proof of Lemma \ref{lem:geom-monotone}]
Let $S$ be any feasible set. By Lemma~\ref{lem:geom-score-global}, $\ordscore$ is the
optimal execution order for the set $S$ under the geometric-horizon model.

Consider now the set $S\cup\{e\}$.
Define an auxiliary execution order $\sigma$ for $S\cup\{e\}$ that:
(i) executes all calls in $S$ in the order $\ordscore$, and
(ii) places $e$ \emph{last} (after all calls in $S$).

Under $\sigma$, adding $e$ cannot delay or affect any call in $S$.
Thus its realized marginal contribution under $\sigma$ is
\[
   r_e\,\mathbf{1}\{v_{i_e}\ge r_e\}\,
   \mathbf{1}\{\text{$e$ is reached before the horizon $n$}\}
   \;\ge\;0
\]
for every realization $(v,n)$.
Hence
\[
   \RevV{n}{\Plan(\sigma,S\cup\{e\})}{v}
   \;\ge\;
   \RevV{n}{\Plan(\ordscore,S)}{v}
\]
for every realization $(v,n)$, and therefore
\[
   \RevUnk{\Plan(\sigma,S\cup\{e\})}
   \;\ge\;
   \RevUnk{\Plan(\ordscore,S)}.
\]

Finally, $\ordscore$ is the optimal execution order for
$S\cup\{e\}$, so
\[
   \RevUnk{\Plan(\ordscore,S\cup\{e\})}
   \;\ge\;
   \RevUnk{\Plan(\sigma,S\cup\{e\})}
   \;\ge\;
   \RevUnk{\Plan(\ordscore,S)}.
\]
Thus $\RevUnk{\Plan(\ordscore,\cdot)}$ is monotone.
\end{proof}

\begin{proof}[Proof of Lemma \ref{lem:geom-submod}]
Let $A\subseteq B$ and $e\notin B$, and let $e'$ be the
(unique) element of $B\setminus A$. As is standard, it suffices to prove the inequality in this special case.

For any subset $S\subseteq\mathcal{C}$ and any $x\in S$, let
\[
   \Pred{\ordscore}{x}{S}
   := \{\,y\in S : y \ordscore x\,\},
   \qquad
   \Succ{\ordscore}{x}{S}
   := \{\,y\in S : x \ordscore y\,\}.
\]

Define the random marginal contribution of adding $e$ to $S$ as
\[
   \Delta_e(S)
   :=
   \RevV{N}{\Plan(S\cup\{e\})}{V}
   -
   \RevV{N}{\Plan(S)}{V}.
\]
Then
\[
   \RevUnk{\Plan(S\cup\{e\})} - \RevUnk{\Plan(S)}
   =
   \mathbb{E}\!\left[\Delta_e(S)\right].
\]

There are two cases depending on the relative position of $e'$ and $e$ in the
order induced by $\ordscore$ on $B$.

\medskip
\noindent\textbf{Case 1: $e'$ precedes $e$ in $\ordscore$.}

Then
\[
   \Pred{\ordscore}{e}{B}
   =
   \Pred{\ordscore}{e}{A} \cup \{e'\},
   \qquad
   \Succ{\ordscore}{e}{B}
   =
   \Succ{\ordscore}{e}{A}.
\]
We condition on whether $e'$ succeeds or fails.

\medskip
\noindent\emph{If $V_{i_{e'}}\ge r_{e'}$:}

Then $e'$ accepts when reached.  In particular, whenever execution would otherwise reach $e$,
it instead terminates at $e'$, and hence $e$ is never reached under $B$.  Thus
$\Delta_e(B)=0$.  Moreover, since valuations are independent across agents and
$e'\notin A\cup\{e\}$, the random variable $\Delta_e(A)$ is independent of the
event $\{V_{i_{e'}}\ge r_{e'}\}$.  Therefore,
\[
\mathbb{E}\!\left[\Delta_e(A)\mid V_{i_{e'}}\ge r_{e'}\right]
=
\mathbb{E}\!\left[\Delta_e(A)\right]
=
\RevUnk{\Plan(A\cup\{e\})}-\RevUnk{\Plan(A)}
\;\ge\;0,
\]
where the last inequality follows from monotonicity of $\RevUnk{\Plan(\cdot)}$.

\medskip
\noindent\emph{If $V_{i_{e'}} < r_{e'}$:}

Then $e'$ fails when reached.  Relative to $A$, the set $B=A\cup\{e'\}$ inserts one additional
failed predecessor before $e$.  Under the geometric horizon, this one-step delay
multiplies the probability of reaching $e$ by $\rho$, and hence
\[
\mathbb{E}\!\left[\Delta_e(B)\mid V_{i_{e'}}<r_{e'}\right]
=
\rho\cdot
\mathbb{E}\!\left[\Delta_e(A)\mid V_{i_{e'}}<r_{e'}\right]
\;\le\;
\mathbb{E}\!\left[\Delta_e(A)\mid V_{i_{e'}}<r_{e'}\right],
\]
where the inequality uses $\rho\in(0,1)$.

\medskip
\noindent\textbf{Case 2: $e'$ succeeds $e$ in $\ordscore$.}

In this case, the relative order of all calls before $e$ is identical under $A$
and $B$.  Let
\[
P := \Pred{\ordscore}{e}{A}=\Pred{\ordscore}{e}{B}.
\]
Thus, whether execution reaches $e$ depends only on the calls in $P$ and is the
same for both sets.

We now condition on whether $e$ accepts or fails.

\medskip
\noindent\emph{If $V_{i_e}\ge r_e$:}

Whenever execution reaches $e$, the call accepts and execution terminates.
Thus, inserting $e$ replaces whatever continuation revenue would
have been earned from calls after $e$ by the fixed reward $r_e$.

Let $\alpha(P)$ denote the probability that execution reaches $e$.  Then, for any
$S\in\{A,B\}$,
\[
\mathbb{E}\!\left[\Delta_e(S)\mid V_{i_e}\ge r_e\right]
=
\alpha(P)\,r_e
-
\rho\cdot
\RevV{N}{\Plan(\Succ{\ordscore}{e}{S})}{V}\mid P.
\]

Since $\Succ{\ordscore}{e}{A}\subseteq \Succ{\ordscore}{e}{B}$, omitting $e$ yields
weakly larger continuation revenue under $B$.  Therefore,
\[
\mathbb{E}\!\left[\Delta_e(B)\mid V_{i_e}\ge r_e\right]
\;\le\;
\mathbb{E}\!\left[\Delta_e(A)\mid V_{i_e}\ge r_e\right].
\]

\medskip
\noindent\emph{If $V_{i_e}< r_e$:}

Whenever execution reaches $e$, the call fails and execution proceeds to the
successor region.  Relative to omitting $e$, inserting $e$ introduces one
additional failed step before reaching the successor calls, which under the
geometric horizon multiplies the continuation value by $\rho$.  Thus, for any
$S\in\{A,B\}$,
\[
\mathbb{E}\!\left[\Delta_e(S)\mid V_{i_e}< r_e\right]
=
(\rho-1)\,
\RevV{N}{\Plan(\Succ{\ordscore}{e}{S})}{V}\mid P.
\]

Since $\Succ{\ordscore}{e}{A}\subseteq \Succ{\ordscore}{e}{B}$ and $\rho-1<0$, it
follows that
\[
\mathbb{E}\!\left[\Delta_e(B)\mid V_{i_e}< r_e\right]
\;\le\;
\mathbb{E}\!\left[\Delta_e(A)\mid V_{i_e}< r_e\right].
\]
In all cases above we have shown that
\[
\mathbb{E}\!\left[\Delta_e(B)\mid \mathcal{E}\right]
\;\le\;
\mathbb{E}\!\left[\Delta_e(A)\mid \mathcal{E}\right]
\]
for the relevant conditioning event $\mathcal{E}$.  Taking expectations over
$\mathcal{E}$ gives $\mathbb{E}[\Delta_e(B)]\le \mathbb{E}[\Delta_e(A)]$, i.e.,
\[
\RevUnk{\Plan(A\cup\{e\})}-\RevUnk{\Plan(A)}
\;\ge\;
\RevUnk{\Plan(B\cup\{e\})}-\RevUnk{\Plan(B)},
\]
as claimed.
\end{proof}

\begin{proof}[Proof of Theorem \ref{thm:geom-reduction}]
Monotonicity follows from Lemma~\ref{lem:geom-monotone}, and submodularity
from Lemma~\ref{lem:geom-submod}.
The single-call feasibility constraint is a partition matroid.
\end{proof}

\section{Appendix - Section \ref{sec:IFR}}
\label{sec:app-ifr}
This appendix contains examples, proof sketches and formal proofs for Section \ref{sec:IFR}.

\refstepcounter{subsection}
\subsection*{Examples}
\label{sec:app-ifr-examples}

\begin{example}[Reverse-time submodular order fails under correlated values]
\label{ex:submodular-order-failure}
Consider three calls $e_1,e_2,e_3$, one available at each time $t=1,2,3$, and a
deterministic horizon ($q_t\equiv 1$). We execute any chosen set in reverse-time
order: time~1 first, then time~2, then time~3. Note that the choice of the calls is such that no pruning is required here, $\RevUnk{\cdot}$ is simply the expected
revenue of this fixed execution rule.

The calls have posted prices
\[
r_{e_1}=70,\qquad r_{e_2}=30,\qquad r_{e_3}=50,
\]
and marginal acceptance probabilities
\[
p_{e_1}=0.25,\qquad p_{e_2}=0.8,\qquad p_{e_3}=0.5.
\]
Assume the acceptance event of $e_3$ is correlated with earlier outcomes so that
\[
\Pr[e_3 \text{ is accepted} \mid e_1 \text{ is rejected}]=0.33,
\qquad
\Pr[e_3 \text{ is accepted} \mid e_1 \text{ is rejected},\, e_2 \text{ is rejected}]=0.33.
\]

\smallskip
\noindent
We compute:
\[
\RevUnk{\{e_3\}} = 50\cdot 0.5 = 25,
\]
\[
\RevUnk{\{e_2,e_3\}}
=30\cdot 0.8 + (1-0.8)\cdot 50\cdot 0.5
=29.
\]
Now add the time-1 call to the smaller set:
\[
\RevUnk{\{e_1,e_3\}}
=70\cdot 0.25 + 0.75\cdot 50\cdot 0.33
=29.875,
\]
so
\[
\RevUnk{\{e_1,e_3\}}-\RevUnk{\{e_3\}}=4.875.
\]
Add the time-1 call to the larger set:
\[
\RevUnk{\{e_1,e_2,e_3\}}
=70\cdot 0.25
+0.75\Bigl(30\cdot 0.8 + 0.2\cdot 50\cdot 0.33\Bigr)
=37.975,
\]
so
\[
\RevUnk{\{e_1,e_2,e_3\}}-\RevUnk{\{e_2,e_3\}}=8.975.
\]
Let $S=\{e_3\}\subset T=\{e_2,e_3\}$ and let $e=e_1$ be the earlier call. Then
\[
\RevUnk{(S\cup\{e\})}-\RevUnk{(S)}
\;<\;
\RevUnk{(T\cup\{e\})}-\RevUnk{(T)},
\]
so the reverse-time submodular-order inequality fails under correlations.

In the independent case, conditioning on earlier rejections does not change the
acceptance probability of any later call. As a result, when we compare the
marginal effect of inserting an earlier call $e$ into two sets $S\subseteq T$,
the only difference between the two marginals comes from the continuation term:
the larger suffix $T$ offers a (weakly) larger continuation value than $S$,
which is exactly what makes the reverse-time submodular-order comparison go
through.

Under correlations, this reasoning breaks. Even though the larger suffix still
offers a larger continuation value in the usual sense, the value of that
continuation is evaluated under a different conditional distribution that
depends on which earlier calls were rejected. In
Example~\ref{ex:submodular-order-failure}, conditioning on an early rejection
strictly reduces the acceptance probability of the time-3 call, so inserting
$e_1$ changes not only the chance of reaching later times but also the
conditional performance of the suffix upon reaching it. This dependence on the
rejection history can make the marginal contribution of $e_1$ larger for the
superset $T$, violating reverse-time submodular order.
\end{example}

\begin{example}[Unbounded gap of relaxation under correlated values]
\label{ex:milp-misleading}
Let $\varepsilon\in(0,1)$ and let $m:=1/\varepsilon$ (assume $m$ is an integer).
Assume a deterministic horizon ($q_t\equiv 1$). Consider the relaxation
\eqref{eq:time-milp}.

There are $m$ \emph{small} calls $e_1,\dots,e_m$ and one \emph{special} call $e_0$.
Set
\[
p_{e_j}=\varepsilon,\qquad r_{e_j}=\frac{1}{\varepsilon}
\quad\text{for } j=1,\dots,m,
\qquad\text{and}\qquad
p_{e_0}=1,\quad r_{e_0}=\sqrt{\frac{1}{\varepsilon}}=\frac{1}{\sqrt{\varepsilon}}.
\]
Let $Z\sim\mathrm{Bernoulli}(\varepsilon)$ and define correlated successes by
\[
X_{e_j}:=Z\ \ \text{for all } j=1,\dots,m,
\qquad\text{and}\qquad
X_{e_0}\equiv 1.
\]
Thus either all small calls succeed (when $Z=1$) or all fail (when $Z=0$), while
the special call always succeeds.

\smallskip
\noindent
\textbf{Relaxation value:}
The relaxation can set $x_{e_j}=1$ and $y_{e_j}=\varepsilon$ for all $j$ and $y_{e_0}=0$.
This is feasible since $\sum_{j=1}^m y_{e_j}=m\varepsilon=1$, and achieves
\[
\sum_{j=1}^m r_{e_j}y_{e_j}
=
m\cdot\frac{1}{\varepsilon}\cdot \varepsilon
=
\frac{1}{\varepsilon}.
\]
Allocating the unit budget to the special call yields at most $r_{e_0}=1/\sqrt{\varepsilon}$,
so the relaxation optimum is at least $1/\varepsilon$.

\smallskip
\noindent
\textbf{True optimal revenue:}
Executing only the special call yields expected revenue $\RevUnk{(\{e_0\})}=1/\sqrt{\varepsilon}$,
so $\mathrm{OPT}\ge 1/\sqrt{\varepsilon}$.
On the other hand, executing the small calls in any order yields a sale iff $Z=1$,
and then the first attempted small call succeeds and pays $1/\varepsilon$, hence
\[
\RevUnk{(e_1,\dots,e_m)}=\varepsilon\cdot\frac{1}{\varepsilon}=1.
\]

\smallskip
\noindent
In particular, the relaxation value can exceed $\mathrm{OPT}$ by a factor $1/\sqrt{\varepsilon}$,
and the plan suggested by the relaxation's optimum (spending its entire probability
budget on the small calls) achieves expected revenue $1$, which is a factor
$1/\sqrt{\varepsilon}$ smaller than $\RevUnk{(e_0)}=1/\sqrt{\varepsilon}$.

\smallskip
\noindent
The relaxation only enforces marginal bounds $y_e\le p_e x_e$ and the single budget
$\sum_e y_e\le 1$, so it can spread probability mass across many calls. Under correlation,
the $m$ small calls share the same success event $Z$, so together they behave like a
single call with success probability $\varepsilon$, which the relaxation fails to recognize.
\end{example}

\begin{example}[No fixed quantile yields a representative horizon]
\label{ex:no-fixed-quantile}
Consider $\alpha\in(0,1)$.  Let the random horizon be
\[
N=\begin{cases}
1 & \text{w.p. }\alpha,\\
H & \text{w.p. }1-\alpha,
\end{cases}
\]
for a large $H$.  Then $q_1=1$ and $q_t=\Pr[N\ge t]=1-\alpha$ for all
$t=2,\ldots,H$, i.e., there is one sure offer opportunity followed by many
opportunities that occur only with probability $\varepsilon:=1-\alpha$.

Consider $m$ buyers, with i.i.d.\ values
$V_i\in\{0,1\}$ where $\Pr[V_i=1]=p$.  Posting price $1$ to any buyer yields
expected revenue $p$ from a single call.  Since the $\alpha$-quantile of $N$
equals $t_\alpha=1$, any policy optimized for the deterministic horizon
$t_\alpha$ makes at most one offer and hence achieves at most $ALG\le p$.

On the other hand, the policy that sequentially offers price $1$ to distinct
buyers for as many steps as available earns conditional revenue
$1-(1-p)^m$ when $N=H$, and thus
\[
OPT \;\ge\; (1-\alpha)\bigl(1-(1-p)^m\bigr).
\]
Choosing $m=\lceil 1/p\rceil$ gives $1-(1-p)^m\ge 1-e^{-1}$, so
$OPT\ge (1-\alpha)(1-e^{-1})$.  Therefore,
\[
\frac{ALG}{OPT}\;\le\;\frac{p}{(1-\alpha)(1-e^{-1})}\xrightarrow[p\to 0]{}0.
\]
Hence selecting a deterministic horizon via any fixed quantile (e.g., the
median) cannot guarantee a constant-factor approximation for arbitrary horizon
distributions.
\end{example}

\begin{example}[Tightness of best deterministic greedy]
\label{ex:best-greedy-tight}
We construct a family of instances on which the best deterministic greedy
solution obtains only an $O(1/\log M)$ fraction of the optimal unknown-horizon
revenue.

Let $M=T$, and let there be one potential offer for each buyer.  For each
$i\in[T]$, define an offer $c_i$ to buyer $i$ at price
\[
    p_i := \Lambda i,
\]
where $\Lambda>0$ is a constant chosen large enough so that
\[
    \sum_{i=1}^{\infty} \frac{1}{\Lambda i^{3/2}} \le 1.
\]
For example, any sufficiently large constant $\Lambda$ works.  The buyers'
valuations are correlated as follows.  There is a null state, in which no buyer
accepts, and for each $i\in[T]$ there is a state $s_i$ with probability
\[
    \Pr[s_i] := \frac{1}{\Lambda i^{3/2}}.
\]
In state $s_i$, buyer $i$ has value $p_i$, while every other buyer has value
zero.  Thus, in every non-null state, exactly one buyer accepts their
corresponding offer.

The expected revenue of offer $c_i$ is therefore
\[
    w_i
    :=
    p_i \Pr[s_i]
    =
    \Lambda i \cdot \frac{1}{\Lambda i^{3/2}}
    =
    \frac{1}{\sqrt{i}}.
\]
Hence
\[
    w_1 \ge w_2 \ge \cdots \ge w_T,
\]
while the prices satisfy
\[
    p_1 < p_2 < \cdots < p_T.
\]
Let the horizon survival probabilities be
\[
    q_t := \frac{1}{\sqrt{t}},
    \qquad t=1,\ldots,T.
\]

Since at most one buyer accepts in any state, the unknown-horizon revenue of a
sequence $\sigma$ is simply
\[
    \RevUnk{\sigma}
    =
    \sum_{t} q_t w_{\sigma_t}.
\]
The optimal policy therefore orders offers by decreasing $w_i$, namely
$1,2,\ldots,T$.  Since both $(q_t)$ and $(w_i)$ are decreasing, the
rearrangement inequality gives
\[
    OPT
    =
    \sum_{i=1}^{T} q_i w_i
    =
    \sum_{i=1}^{T} \frac{1}{i}
    =
    H_T
    =
    \Theta(\log T).
\]

Now consider the deterministic greedy solution for a fixed horizon $h$.  The
greedy routine selects the $h$ offers with largest expected revenues, namely
$c_1,\ldots,c_h$.  However, because prices are increasing in the index, its
decreasing-price order is
\[
    c_h,c_{h-1},\ldots,c_1.
\]
Thus, when this horizon-$h$ greedy solution is evaluated under the unknown
horizon, its revenue is
\[
    G(h)
    =
    \sum_{i=1}^{h} q_i w_{h+1-i}
    =
    \sum_{i=1}^{h}
    \frac{1}{\sqrt{i(h+1-i)}}.
\]
We now show that $G(h)=O(1)$ uniformly over all $h$.  By symmetry,
\[
    G(h)
    \le
    2 \sum_{i=1}^{\lceil h/2\rceil}
    \frac{1}{\sqrt{i(h+1-i)}}.
\]
For every $i\le \lceil h/2\rceil$, we have
$h+1-i \ge (h+1)/2$, and hence
\[
    \frac{1}{\sqrt{i(h+1-i)}}
    \le
    \sqrt{\frac{2}{h+1}} \cdot \frac{1}{\sqrt{i}}.
\]
Therefore,
\[
    G(h)
    \le
    2\sqrt{\frac{2}{h+1}}
    \sum_{i=1}^{\lceil h/2\rceil} \frac{1}{\sqrt{i}}
    \le 4,
\]
where the last inequality uses
$\sum_{i=1}^{m} i^{-1/2}\le 2\sqrt{m}$.

Thus every deterministic greedy candidate has unknown-horizon revenue at most a
constant:
\[
    \max_{h\le T} G(h) = O(1).
\]
On the other hand, $OPT=\Theta(\log T)$.  Hence
\[
    \frac{\RevUnk{\pi^{\mathrm{BG}}}}{OPT}
    =
    \frac{\max_{h\le T}G(h)}{OPT}
    =
    O\!\left(\frac{1}{\log T}\right)
    =
    O\!\left(\frac{1}{\log M}\right).
\]
Together with the guarantee in Theorem~\ref{thm:best-greedy-finite}, this shows
that the approximation ratio of
Algorithm~\ref{alg:best-greedy-horizon} is $\Theta(1/\log M)$.
\end{example}

\refstepcounter{subsection}
\subsection*{Formal Proofs}
\label{sec:app-ifr-proofs}

To relate the random-horizon objective to deterministic benchmarks, we define the
optimal deterministic revenue curve
\[
  \OPT_n
  := \max_{\substack{S\subseteq\Calls:\\ |S|\le n,\\
  \forall i\in I:\ |\{e\in S:\ i_e=i\}|\le 1}}
  \RevN{n}{\pi^{\mathrm{d}}(S)} ,
\]
which captures the maximum achievable revenue when exactly \(n\) calls can be
attempted.

\begin{lemma}[Scaling of the deterministic revenue curve]
\label{lem:g-scaling}
The function \(\OPT_n\) is nondecreasing, and for every \(m\in\mathbb{N}\) and every
integer \(k\ge 1\),
$
  \OPT_{km} \;\le\; k\,\OPT_m.
$
\end{lemma}
\begin{proof}[Proof of Lemma \ref{lem:g-scaling}]
Monotonicity is immediate since enlarging the budget enlarges the feasible
family.

Consider some \(m\) and \(k\), and let \(S^\star\) be an optimizer for \(\OPT_{km}\). Partition
\(S^\star\) into \(k\) (possibly empty) subsets \(S_1,\ldots,S_k\) with
\(|S_j|\le m\) for all \(j\). Each \(S_j\) satisfies the one-call-per-agent
constraint because \(S^\star\) does.

Let \(U_j:=\bigcup_{\ell=1}^j S_\ell\) with \(U_0=\emptyset\), and let
\(F(U):=\RevN{km}{\pi^{\mathrm{d}}(U)}\). Since \(F\) is monotone and submodular, we have the inequality
\(F(S_j\mid U_{j-1})\le F(S_j)\). Therefore,
\[
  \OPT_{km}
  \;=\;
  F(S^\star)
  \;=\;
  \sum_{j=1}^k \bigl(F(U_j)-F(U_{j-1})\bigr)
  \;\le\;
  \sum_{j=1}^k F(S_j).
\]
Finally, \(|S_j|\le m\) implies \(F(S_j)=\RevN{km}{\pi^{\mathrm{d}}(S_j)}
=\RevN{m}{\pi^{\mathrm{d}}(S_j)}\le \OPT_m\). Hence \(\OPT_{km}\le k\,\OPT_m\).
\end{proof}

\begin{lemma}[IFR tail bound]
\label{lem:ifr-blocks}
Let $N$ be a nonnegative integer-valued random variable with an increasing
failure rate, and let $m$ be a (left) median of $N$, i.e.,
$\Pr[N\ge m]\ge\tfrac{1}{2}$ and $\Pr[N\ge m{+}1]\le\tfrac{1}{2}$.
Let $q(t):=\Pr[N\ge t]$ be its survival function. Then for every integer $k\ge1$,
\[
  q(km) \;\le\; 2^{-k}, \qquad
  \mathbb{E}\bigl[\lceil N/m\rceil\bigr] \;\le\; 2.
\]
\end{lemma}
\begin{proof}[Proof of Lemma \ref{lem:ifr-blocks}]
IFR implies that the discrete hazard
$h(t):=\Pr[N=t\mid N\ge t]=\frac{q(t)-q(t+1)}{q(t)}$ is nondecreasing, which
is equivalent to:
\[
  q(t+s) \;\le\; q(t)\,q(s)
  \qquad\forall\,t,s\ge0.
\]
Since $m$ is a median and $q$ is nonincreasing, we have $q(m)\le\tfrac{1}{2}$.
By induction,
\[
  q(km)
  \;\le\;
  q((k{-}1)m)\,q(m)
  \;\le\;
  q(m)^k
  \;\le\;
  2^{-k},
  \qquad k\ge1.
\]

For the second claim, we use the tail-sum representation for integer-valued
nonnegative random variables:
\[
  \mathbb{E}\bigl[\lceil N/m\rceil\bigr]
  = \sum_{r=1}^{\infty} \Pr\bigl[\lceil N/m\rceil \ge r\bigr]
  = \sum_{r=1}^{\infty} \Pr\bigl[N \ge (r{-}1)m + 1\bigr].
\]
For each $r\ge1$,
\[
  \Pr\bigl[N \ge (r{-}1)m + 1\bigr]
  \;\le\;
  \Pr[N\ge (r{-}1)m]
  = q((r{-}1)m)
  \;\le\;
  2^{-(r-1)}.
\]
Thus
\[
  \mathbb{E}\bigl[\lceil N/m\rceil\bigr]
  \;\le\;
  \sum_{r=1}^{\infty} 2^{-(r-1)}
  = 2.
\]
\end{proof}

\begin{proof}[Proof of Theorem \ref{thm:ifr-median-greedy}]
By Corollary~\ref{cor:known-horizon},
\begin{equation}
  \RevN{m}{\Plan^{\mathrm{med}}}
  \;\ge\;
  \Bigl(1-\tfrac1e\Bigr)\,\OPT_m.
  \label{eq:med-det-approx}
\end{equation}

For any realization \(N=n\), executing \(\Plan^{\mathrm{med}}\) under horizon \(n\)
is equivalent to executing it under horizon \(\min\{n,m\}\), since
\(\Plan^{\mathrm{med}}\) contains at most \(m\) calls. Hence
\[
  \RevUnk{\Plan^{\mathrm{med}}}
  \;=\;
  \mathbb{E}_N\!\bigl[\RevN{\min\{N,m\}}{\Plan^{\mathrm{med}}}\bigr]
  \;\ge\;
  \RevN{m}{\Plan^{\mathrm{med}}}\cdot \Pr[N\ge m]
  \;\ge\;
  \tfrac12\,\RevN{m}{\Plan^{\mathrm{med}}},
\]
where the last inequality uses that \(m\) is a median. Combining with
\eqref{eq:med-det-approx} gives
\begin{equation}
  \RevUnk{\Plan^{\mathrm{med}}}
  \;\ge\;
  \tfrac12\Bigl(1-\tfrac1e\Bigr)\,\OPT_m.
  \label{eq:med-lower}
\end{equation}

Next, we upper bound \(OPT^{\mathrm{IFR}}\) in terms of \(\OPT_m\).
For any plan \(\Plan\) and any realization \(N=n\),
\(\RevN{n}{\Plan}\le \OPT_n\), so \(OPT^{\mathrm{IFR}}\le \mathbb{E}[\OPT_N]\).
By Lemma~\ref{lem:g-scaling}, for every \(n\),
\[
  \OPT_n
  \;\le\;
  \OPT_{\lceil n/m\rceil\,m}
  \;\le\;
  \lceil n/m\rceil\,\OPT_m.
\]
Taking expectations and applying Lemma~\ref{lem:ifr-blocks},
\[
  OPT^{\mathrm{IFR}}
  \;\le\;
  \mathbb{E}[\OPT_N]
  \;\le\;
  \mathbb{E}\bigl[\lceil N/m\rceil\bigr]\,\OPT_m
  \;\le\;
  2\,\OPT_m.
\]
Substituting this into \eqref{eq:med-lower} yields
\[
  \RevUnk{\Plan^{\mathrm{med}}}
  \;\ge\;
  \tfrac12\Bigl(1-\tfrac1e\Bigr)\,\OPT_m
  \;\ge\;
  \tfrac14\Bigl(1-\tfrac1e\Bigr)\,OPT^{\mathrm{IFR}}.
\]
\end{proof}

\begin{proof}[Proof of Lemma \ref{lem:nonignorable-count}]
Consider \(k\in\mathcal{K}_{\mathrm{heavy}}\). Then \(V_k(\pi^\star)\ge (1-\varepsilon)^k V_0^\star\).

Let \(\pi_0^\star\) be a maximizer attaining \(V_0^\star\). Since \(q_t>1/2\) for
all \(t\in B_0\), we have
\[
  V_0^\star
  \;=\;
  \sum_{t\in B_0} q_t\Big(\RevN{t}{\pi_0^\star}-\RevN{t-1}{\pi_0^\star}\Big)
  \;\ge\;
  \frac12 \sum_{t\in B_0}\Big(\RevN{t}{\pi_0^\star}-\RevN{t-1}{\pi_0^\star}\Big)
  \;\ge\;
  \frac12\,\RevN{n_0}{\pi_0^\star}.
\]
Hence
\begin{equation}
\label{eq:v0star-lower-det}
  \RevN{n_0}{\pi_0^\star} \;\le\; 2V_0^\star.
\end{equation}

On the other hand, for bucket \(k\), we have \(q_t\le 2^{-k}\) for every \(t\in B_k\),
and therefore
\[
  V_k(\pi^\star)
  \;=\;
  \sum_{t\in B_k} q_t\Big(\RevN{t}{\pi^\star}-\RevN{t-1}{\pi^\star}\Big)
  \;\le\;
  2^{-k}\,\RevN{n_k}{\pi^\star}.
\]
Combining with \(V_k(\pi^\star) \ge (1-\varepsilon)^k V_0^\star\) and \eqref{eq:v0star-lower-det} gives
\[
  2^{-k}\,\RevN{n_k}{\pi^\star}
  \;\ge\;
  V_k(\pi^\star)
  \;\ge\;
  (1-\varepsilon)^k V_0^\star
  \;\ge\;
  (1-\varepsilon)^k \cdot \frac12\,\RevN{n_0}{\pi_0^\star}.
\]
Multiplying by \(2^k\) yields
\begin{equation}
\label{eq:det-growth}
  \RevN{n_k}{\pi^\star}
  \;\ge\;
  \frac12(2-2\varepsilon)^k\,\RevN{n_0}{\pi_0^\star}.
\end{equation}

Finally, because there are \(M\) agents and at most one call per agent, any plan
has length at most \(M\), so \(\RevN{n}{\pi}\le \RevN{M}{\pi}\) for all \(n\) and \(\pi\).
Taking the maximum over \(\pi\) and applying Lemma~\ref{lem:g-scaling} with \(m=1\)
gives
\[
  \RevN{n_k}{\pi^\star}
  \;\le\;
  \max_{\pi}\RevN{M}{\pi}
  \;\le\;
  M\cdot \max_{\pi}\RevN{1}{\pi}
  \;\le\;
  M\cdot \max_{\pi}\RevN{n_0}{\pi}
  \;\le\;
  M\cdot \RevN{n_0}{\pi_0^\star},
\]
where the last inequality uses that \(\pi_0^\star\) is feasible for horizon \(n_0\)
and achieves \(\RevN{n_0}{\pi_0^\star}\).

Combining this upper bound with \eqref{eq:det-growth} yields
\[
  \frac12(2-2\varepsilon)^k\,\RevN{n_0}{\pi_0^\star}
  \;\le\;
  \RevN{n_k}{\pi^\star}
  \;\le\;
  M\cdot \RevN{n_0}{\pi_0^\star},
\]
so \((2-2\varepsilon)^k \le 2M\), i.e.,
\(k \le \log_{2-2\varepsilon}(2M)\). This implies the stated
\(O(\log M)\) bound on \(|\mathcal{K}_{\mathrm{heavy}}|\).
\end{proof}

\begin{proof}[Proof of Lemma \ref{lem:bucket-det}]
Let a bucket \(B_k\)'s indices in increasing order be
\(B_k=\{t_{k,1}<\cdots<t_{k,n_k}\}\).
Recall that \(2^{-(k+1)} < q_t \le 2^{-k}\) for all \(t\in B_k\).
For any plan \(\pi\), the bucket contribution satisfies the bounds
\begin{equation}
\label{eq:bucket-sandwich-clean}
  2^{-(k+1)}\,\RevN{n_k}{\pi}
  \;\le\;
  V_k(\pi)
  \;\le\;
  2^{-k}\,\RevN{n_k}{\pi}.
\end{equation}

Taking the maximum over \(\pi\) in the upper bound of
\eqref{eq:bucket-sandwich-clean} gives
\begin{equation}
\label{eq:Vkmax-upper-clean}
  \max_{\pi} V_k(\pi)
  \;\le\;
  2^{-k}\cdot \max_{\pi}\RevN{n_k}{\pi}.
\end{equation}
Next, by the \((1-1/e)\) guarantee for \textsc{Greedy} on the deterministic-horizon
problem of length \(n_k\) (Corollary~\ref{cor:known-horizon}),
\begin{equation}
\label{eq:greedy-det-clean}
  \RevN{n_k}{\pi_k}
  \;\ge\;
  \Bigl(1-\frac1e\Bigr)\cdot \max_{\pi}\RevN{n_k}{\pi}.
\end{equation}
Applying the lower bound in \eqref{eq:bucket-sandwich-clean} to \(\pi_k\) and
then using \eqref{eq:greedy-det-clean} yields
\[
  V_k(\pi_k)
  \;\ge\;
  2^{-(k+1)}\,\RevN{n_k}{\pi_k}
  \;\ge\;
  2^{-(k+1)}\Bigl(1-\frac1e\Bigr)\max_{\pi}\RevN{n_k}{\pi}.
\]
Combining with \eqref{eq:Vkmax-upper-clean} gives
\[
  V_k(\pi_k)
  \;\ge\;
  \frac12\Bigl(1-\frac1e\Bigr)\cdot \max_{\pi}V_k(\pi),
\]
which proves the claim.
\end{proof}

\begin{proof}[Proof of Theorem \ref{thm:best-greedy-finite}]
Recall that \(\pi^\star\in\arg\max_{\pi}\RevUnk{\pi}\) is an optimal plan, so that
\(OPT=\RevUnk{\pi^\star}\).
Therefore,
\[
  OPT
  \;=\;
  \RevUnk{\pi^\star}
  \;=\;
  \sum_k V_k(\pi^\star)
  \;\le\;
  \sum_k \max_{\pi} V_k(\pi).
\]

We split the indices into light and heavy buckets, and define
\[
  V^\star \;:=\; \max_k \max_{\pi} V_k(\pi).
\]

For light buckets \(k\ge 1\), we have
\(\max_{\pi}V_k(\pi) < (1-\varepsilon)^k V_0^\star\), and hence
\[
  \sum_{k \notin \mathcal{K}_{\mathrm{heavy}}} \max_{\pi}V_k(\pi)
  \;\le\;
  \sum_{k\ge 1}(1-\varepsilon)^k\,V_0^\star
  \;=\;
  \frac{1-\varepsilon}{\varepsilon}\,V_0^\star
  \;\le\;
  \frac{1-\varepsilon}{\varepsilon}\,V^\star.
\]
For the heavy buckets,
\[
  \sum_{k\in\mathcal{K}_{\mathrm{heavy}}} \max_{\pi}V_k(\pi)
  \;\le\;
  |\mathcal{K}_{\mathrm{heavy}}|\,V^\star
  \;\le\;
  O(\log M)\cdot V^\star,
\]
where the last inequality is Lemma~\ref{lem:nonignorable-count}. Combining the
two yields
\[
  OPT
  \;\le\;
  \sum_k \max_{\pi}V_k(\pi)
  \;\le\;
  O(\log M)\cdot V^\star.
\]

Now let \(k^\star\in\arg\max_k V_k(\pi_k)\) be the highest value bucket. Since \(\RevUnk{\pi}=\sum_k V_k(\pi)\) for any plan \(\pi\),
we have
\[
  \RevUnk{\pi^{\mathrm{BB}}}
  \;\ge\;
  V_{k^\star}(\pi_{k^\star})
  \;=\;
  \max_k V_k(\pi_k).
\]
By Lemma~\ref{lem:bucket-det},
\[
  \max_k V_k(\pi_k)
  \;\ge\;
  \frac12\Bigl(1-\frac1e\Bigr)\max_k \max_{\pi}V_k(\pi)
  \;=\;
  \frac12\Bigl(1-\frac1e\Bigr)\,V^\star.
\]
Combining with \(OPT \le O(\log M)\cdot V^\star\) gives
\[
  \RevUnk{\pi^{\mathrm{BB}}}
  \;\ge\;
  \Omega\!\left(\frac{1}{\log M}\right)\cdot OPT,
\]
as claimed.
\end{proof}

\begin{algorithm}[t]
\caption{Best Greedy over Candidate Horizons}
\label{alg:best-greedy-horizon}
\DontPrintSemicolon
\SetKwInOut{Input}{Input}
\SetKwInOut{Output}{Output}

\Input{call set $\mathcal{C}$; survival probabilities $(q_t)_{t=1}^T$}
\Output{a plan $\pi^{\mathrm{BG}}$}

\BlankLine
\For{$h \gets 1$ \KwTo $T$}{
    $\pi_h \gets \textsc{Greedy}(\mathcal{C},h)$\;

    $\mathrm{VAL}_h \gets
    \RevUnk{\pi_h}
    =
    \sum_{t=1}^{T}
    q_t\bigl(\RevN{t}{\pi_h}-\RevN{t-1}{\pi_h}\bigr)$\;
}

\BlankLine
$h^\star \in \arg\max_{\{h \in [T] : q_h > 0\}} \mathrm{VAL}_h$\;

\BlankLine
\Return{$\pi^{\mathrm{BG}} \gets \pi_{h^\star}$}\;

\end{algorithm}

\end{document}